\documentclass[12pt]{amsart}

\usepackage{mathtools}
\usepackage{amsmath}
\usepackage{graphicx}
\usepackage[all]{xy}
\usepackage{todonotes}
\usepackage{enumerate}
\usepackage{verbatim}
\usepackage{multirow}
\usepackage{pgf,tikz}
\usepackage{longtable}
\usepackage{wrapfig}

\usepackage{algcompatible}
\usepackage{algorithm}
\usepackage{algpseudocode}

\usepackage{bigstrut}
\usepackage[english]{babel}
\usepackage{hyperref}
\usepackage{amsmath,amssymb,amsthm,textcomp,mathrsfs,latexsym}
\usepackage{lmodern}
\usepackage{multicol}
\usepackage{mdwlist}
\usepackage{pdfpages}
\usepackage{array}
\usepackage{amscd}
\usepackage{fancyvrb}

\usepackage{microtype}
\usepackage{fancyhdr}
\usepackage{setspace} 
\usepackage{tikz}
\usepackage[left=3cm,top=2.5cm,right=3cm]{geometry}
\usetikzlibrary{arrows}
\usetikzlibrary{matrix}
\usetikzlibrary{shapes.geometric}
\usetikzlibrary{calc}

\hypersetup{
    colorlinks=true,       
    linkcolor=blue,          
    citecolor=blue,        
    filecolor=blue,      
    urlcolor=blue           
}

\usepackage{color}
\definecolor{red-}{rgb}{1.0,0.0,0.0}
\definecolor{green-}{rgb}{0.0,0.4,0.0}
\definecolor{blue-}{rgb}{0.0,0.0,1.0}

\usepackage{hyperref}
\hypersetup{
    colorlinks=true,
    linkcolor=blue,
    citecolor=blue,
    filecolor=blue,
    urlcolor=blue
}

\usepackage[breakable,skins]{tcolorbox}
\newtcolorbox{tbox}[1][]{%
    breakable,
    enhanced,
    colframe=blue,
    coltitle=white,
    #1
}

\theoremstyle{theorem}
\newtheorem{theorem}{Theorem}[section]
\newtheorem{corollary}[theorem]{Corollary}
\newtheorem{lemma}[theorem]{Lemma}
\newtheorem{proposition}[theorem]{Proposition}

\theoremstyle{definition}
\newtheorem{remark}[theorem]{Remark}

\newtheorem{prog}{Program}
\newtheorem{definition}[theorem]{Definition}

\newtheorem{example}[theorem]{Example}

\begin{document}

\title[Resultants of skew polynomials]{Resultants of skew polynomials over division rings}

\author[A.E. Almendras, J.A. Briones, A.L. Tironi]{Alexis Eduardo Almendras Valdebenito, Jonathan Armando Briones Donoso and Andrea Luigi Tironi}

\date{\today}

\address{
Universidad de Concepci\'on,
Facultad de Ciencias F\'isicas y Matem\'aticas,
Departamento de Matem\'atica,
Casilla 160-C,
Concepci\'on, Chile}
\email{atironi@udec.cl}
\email{jonathanbriones@udec.cl}

\address{
Universidad de Concepci\'on,
Escuela de Educaci\'on,
Departamento de Ciencias B\'asicas,
Los \'Angeles, Chile}
\email{alexisalmendras@udec.cl}



\begin{abstract}
Let $\mathbb{F}$ be a division ring. We extent some of the main well-known results about the resultant of two univariate polynomials to the more general context of an Ore extension $\mathbb{F}[x;\sigma,\delta]$. Finally, some algorithms and Magma programs are given as a numerical application of the main theoretical results of this paper.
\end{abstract}

\maketitle

\section{Introduction}\label{intro}

In commutative algebra, the notion of resultant (or eliminant) of two univariate polynomials defined over a field is well-known and classical and many results about it can be found in literature (e.g, see \cite{ModernAlg}, \cite{resultants} or \cite{Alglang}). The classical resultant of two polynomials is in fact a polynomial expression of their coefficients, which is equal to zero if and only if the polynomials have a common root, possibly in a field extension, or equivalently, a common factor over their field of coefficients. The resultant is widely used in number theory, algebraic geometry, symbolic integration, computer algebra, and it is a built-in function of most computer algebra systems. The resultant of two univariate polynomials over a field, or a commutative ring, is commonly defined as the determinant of their Sylvester matrix. More precisely, let
$p(x)=p_rx^r+\dots+p_1x+p_0$ and $q(x)=q_sx^s+\dots+q_1x+q_0$ be two non-zero polynomials with $p_r\neq 0,q_s\neq 0$. The map $\varphi : \mathcal{P}_s\times\mathcal{P}_r\to\mathcal{P}_{r+s}$ given by $\varphi (a,b)=ap+bq$ is a linear map between two vector spaces of the same dimension, where $\mathcal{P}_i$ is the vector space of dimension $i$ whose elements are the polynomials of degree less than $i$. Over the basis of the powers of the variable $x$, the above map $\varphi$ is represented by a square matrix of dimension $r+s$, which is called the Sylvester matrix of $p$ and $q$.  

Let $\mathbb{F}[x;\sigma,\delta]$ be an Ore extension (Definition \ref{ringskew}), where $\mathbb{F}$ is a division ring, $\sigma :\mathbb{F}\to\mathbb{F}$ is a ring homomorphism and $\delta :\mathbb{F}\to\mathbb{F}$ is a $\sigma$-derivation (Definition \ref{def1}). Inspired by \cite{ResEric}, the main purpose of this paper is to extend in $\mathbb{F}[x;\sigma,\delta]$ all the above results and the well-known criteria equivalent to the condition that the resultant of two univariate skew polynomials is equal to zero. Finally, along this paper, we give also some algorithms of the main algebraic computations and some of their respective Magma programs \cite{Magma}. 

\smallskip

The paper is organized as follows. In Section 1, we recall some basic definitions and notation, the main properties of $\mathbb{F}[x;\sigma,\delta]$ and we give some preliminary results. In Section 2, after some technical lemmas, we first introduce the right $(\sigma,\delta)$-resultant $R_{\mathbb{F}}^{\sigma,\delta}(f,g)$ of two skew polynomials $f,g\in\mathbb{F}[x;\sigma,\delta]$ (Definition \ref{Res1.3}) and, after some of its properties (see Propositions \ref{basicpropert} and \ref{pro3.19}), we then prove three results about some equivalent conditions which give $R_{\mathbb{F}}^{\sigma,\delta}(f,g)=0$ (see Theorems \ref{Res1.5}, \ref{Res1.15} and Proposition \ref{Res1.18}) and a characterization of the degree of the greatest common right divisor $gcrd(f,g)$ (see Theorem \ref{teor 3.7}), which can be also applied to check when $gcrd(f,g)\neq 1$ . Moreover, we introduce the notion of the left $(\sigma,\delta)$-resultant of two skew polynomials (Definition \ref{leftresultant}), rewriting in this context some of the main previous results, and we give some equivalent conditions to the fact that a skew polynomial admits a right or left root of positive multiplicity (see Theorems \ref{Res2.2} and \ref{thm-multiplicity}). Finally, through the paper, we give some algorithms and their respective Magma programs as computational applications of the main algebraic results which allowed us to construct all the examples of this paper in a very simple manner.

%
%
%


\section{Background material and preliminary results}\label{S.1.1}

Denote by $\mathbb{F}$ a division ring (or a skew field), that is, a unitary ring (not necessarily commutative) in which every non-zero element is invertible in $\mathbb{F}$. Evidently, every field is a division ring. The most familiar
example of a division ring which is not a field is the ring $\mathbb{H}$ of Hamilton's
quaternions. However, also there are interesting methods for constructing non-commutative division rings (e.g., if $R$ is a ring and $S$ is a simple module over $R$, then the endomorphisms ring of $S$ is always a division ring \cite[Lemma 3.6]{lambook}). On the other hand, if we assume that $\mathbb{F}$ is a finite division ring, then it is known that $\mathbb{F}$ is a finite field (see \cite[p.~250]{herstein}). 
\bigskip

To define skew polynomial rings (Ore extensions), we begin by introducing the concept of $\sigma$-derivation.

\begin{definition}\label{def1}
Let $\sigma: \mathbb{F} \to \mathbb{F}$ be a non-zero ring endomorphism. An additive group homomorphism $\delta: \mathbb{\mathbb{F}}\to \mathbb{F}$ is called a
\textit{$\sigma$-derivation} (over $\mathbb{F}$) if
$$
\delta(ab) = \sigma(a)\delta(b) + \delta(a)b 
$$
for every $a, b \in \mathbb{F}$.
\end{definition}

\begin{example}
Let $\sigma:\mathbb{F} \to \mathbb{F}$ be a ring homomorphism and let $\beta \in \mathbb{F}$. The map $\delta_{\beta}:\mathbb{F}\to \mathbb{F}$ defined by $$\delta_{\beta}(a):=\sigma(a)\beta-\beta a$$
for all $a\in \mathbb{F}$ is a $\sigma$-derivation. These kind of derivations are called in literature simply \textit{inner derivations}.
\end{example}

\begin{remark}\label{Pre2}
From Definition \ref{def1}, it follows that $\delta(1)=\delta(-1)=0$. Furthermore, $\sigma$ is always a monomorphism but, in general, it is not an automorphism. For instance, in $\mathbb{F}_{p}(t):=\left\{\frac{f}{g}:f,g\in \mathbb{F}_{p}[t], g\neq 0\right\}$ (field of rational functions in the variable $t$ over the finite field $\mathbb{F}_{p}$ with $p$ prime), the endomorphism $\phi:\mathbb{F}_{p}(t) \to \mathbb{F}_{p}(t), x \mapsto x^{p}$ is not surjective because $\phi(\mathbb{F}_{p}(t))$ does not contain $t$.
\end{remark}

The role played by inner derivations when $\mathbb{F}$ is a field is shown by the following result, as indicated in \cite[Section 8.3]{cohnfree} (see also \cite[Proposition 39]{reedsalomoncodes}).

\begin{proposition} \label{inner derivations}
Let $\mathbb{F}$ be a field and consider an endomorphism $\sigma:\mathbb{F}\to \mathbb{F}$ and a $\sigma$-derivation $\delta:\mathbb{F}\to \mathbb{F}$. If $\sigma \neq {Id}$ (the identity automorphism), then $\delta$ is an inner derivation.
\end{proposition}

\begin{remark}
From Proposition \ref{inner derivations}, it is natural to ask how are the ${Id}$-derivations over a field. With respect to this question, we can distinguish two cases. If $\mathbb{F}$ is a finite field then the unique $Id$-derivation is $\delta=0$ (see \cite[Proposition 44]{reedsalomoncodes}). However, if $\mathbb{F}$ is infinite, then it is possible to define non-zero ${Id}$-derivations. For instance, the formal derivation with respect to the variable $t$, given by $\frac{d}{dt}$, is an ${Id}$-derivation over $\mathbb{F}_p(t)$.
\end{remark} 

Following some Ore's ideas in \cite{Ore}, we recall the ring of univariate skew polynomials.

\begin{definition}\label{ringskew}
Given a ring endomorphism $\sigma:\mathbb{F}\to \mathbb{F}$ and a $\sigma$-derivation $\delta:\mathbb{F}\to \mathbb{F}$, we define the \textit{univariate skew polynomial ring}, corresponding to $\sigma$ and $\delta$ and denoted by 
$$\mathcal{R}:=\mathbb{F}[x;\sigma,\delta] \ ,$$ 
as the set of all polynomials $\sum_{i}a_ix^{i}$ ($a_i \in \mathbb{F}$) with the usual sum of polynomials and the product defined accordingly to the following rule
\begin{eqnarray}
xa = \sigma(a)x+ \delta(a)\label{1.1}
\end{eqnarray}
for all $a\in \mathbb{F}$.
\end{definition}

\begin{example} \label{ejF_4}
Consider $\mathbb{F}_4[x;\sigma,\delta_{\alpha}]$ with $\mathbb{F}_4=\{ 0,1,\alpha, \alpha^2\}$, where $\alpha^2+\alpha+1=0$,  $\sigma(a)=a^2$ and $\delta_{\alpha}(a)=\alpha(\sigma(a)+a)$ for all $a \in \mathbb{F}_4$. Then we have 
$$\alpha x\cdot\alpha^2x=\alpha(x \alpha^2)x=\alpha (\sigma(\alpha^2)x+\delta_{\alpha}(\alpha^2))x=\alpha^2x^2+\alpha^2x \ ,$$
$$\alpha^2x \cdot \alpha x=\alpha^2(x\alpha)x=\alpha^2 (\sigma(\alpha)x+\delta_{\alpha}(\alpha))x= \alpha x^2+x \ .$$
\end{example}

\noindent The previous example shows that in general the product of skew polynomials is not commutative and that the product of two monomials is not a monomial. Moreover, let us recall here some basic properties of $\mathcal{R}$:

\begin{itemize}
\item[1.] Let $f(x)=\displaystyle\sum_{i=1}^{m}a_ix^{i}, g(x)=\displaystyle\sum_{j=1}^{n}b_jx^{j} \in \mathcal{R}$ of degrees $m$ and $n$, respectively. By \eqref{For1.3}, we get $f(x)g(x)=...+a_m\sigma^{m}(b_n)x^{m+n}$ and  $a_m\sigma^{m}(b_n) \neq 0$ (because $a_m,b_n \neq 0$ and $\sigma$ is a monomorphism). In particular, we have $\deg (fg) = \deg(f) + \deg (g)$ which implies that $\mathcal{R}$ has not zero divisors.
\item[2.] The Euclidean algorithm holds for \textit{right division} (see \cite[p.~483]{Ore}). For any $f(x),g(x) \in \mathcal{R}$ with $g(x) \neq 0$, there are unique $q(x),r(x)\in\mathcal{R}$ such that $$f(x) = q(x) g(x) + r(x)$$ with either $\deg(r) < \deg (g)$, or $r(x) = 0$. For instance, in $\mathbb{F}_4[x;\sigma,\delta]$ with $\sigma,\delta$ defined as in Example \ref{ejF_4}, if we divide $x^3$ by $\alpha x$, unlike the usual division algorithm, the action of $\delta$ in general makes the quotient is not a monomial, but a polynomial. In fact, we have $x^3=(\alpha^2x^2+x+\alpha)(\alpha x)$. 
\item[3.] An important consequence of the right division algorithm is that $\mathcal{R}$ is a left principal ideal domain (LPID), i.e. any left ideal $I \subset \mathcal{R}$ has the form $\mathcal{R}g$, where $g \in \mathcal{R}$ is a polynomial of minimal degree among non-zero elements
of $I$. However, it is also widely known that if $\sigma$ fails to be an automorphism of $\mathbb{F}$, i.e. $\sigma(\mathbb{F}) \neq \mathbb{F}$, then the left division does not work (see \cite[Theorem 6]{Ore}) showing that $\mathcal{R}$ is not in general a right principal ideal domain (RPID).
\end{itemize}

\noindent Consider now monomials $ax^{i},bx^{j},x^{i}\beta,x^{j}\alpha\in\mathcal{R}$. Motivated by Example \ref{ejF_4}, we will give formulas to calculate easily the products $ax^{i}\cdot bx^{j}$ and $x^{i}\beta\cdot x^{j}\alpha$.

\begin{definition}\label{Def Cij}
Let $a\in \mathbb{F}$. We define $\mathcal{C}_{d,s}(a)$ as the sum of all possible compositions (as functions) of $d$ copies of $\delta$ and $s$ copies of $\sigma$ evaluated in $a$ when $(d,s)\in \mathbb{Z}_{\geq 0}\times \mathbb{Z}_{\geq 0}\setminus (0,0)$, $\mathcal{C}_{0,0}(a)=a$ and $\mathcal{C}_{d,s}(a)=0$ otherwise. Moreover, if $\sigma$ is an automorphism, we define $\mathcal{T}_{t,r}(a)$ as the sum of all possible compositions (as functions) of $t$ copies of $\delta\sigma^{-1}$ and $r$ copies of $\sigma^{-1}$ evaluated in $a$ when $(t,r)\in \mathbb{Z}_{\geq 0}\times \mathbb{Z}_{\geq 0}\setminus (0,0)$, $\mathcal{T}_{0,0}(a)=a$ and $\mathcal{T}_{t,r}(a)=0$ otherwise.
\end{definition}

\begin{remark}\label{remark 2.9}
From Definition \ref{Def Cij}, for all $a\in \mathbb{F}$ and $(d,s)\in \mathbb{Z}\times \mathbb{Z}$ we have 
$$\mathcal{C}_{d,s}(a)=\delta( \mathcal{C}_{d-1,s}(a))+\sigma(\mathcal{C}_{d,s-1}(a))\ , \ \mathcal{T}_{d,s}(a)=\delta\sigma^{-1}( \mathcal{T}_{d-1,s}(a))+\sigma^{-1}(\mathcal{T}_{d,s-1}(a))\ .$$
\end{remark}

\begin{lemma}\label{LemaCij} Given a non-negative integer $i$ and $a\in \mathbb{F}$, we have
\begin{eqnarray}
  x^{i}a=\sum_{k=0}^{i}\mathcal{C}_{k,i-k}(a)x^{i-k} \ . \label{1.2}
\end{eqnarray}
Moreover, if $\sigma$ is an automorphism, we get
\begin{eqnarray}
  ax^{i}=\sum_{k=0}^{i}x^{i-k}(-1)^{k}\mathcal{T}_{k,i-k}(a). \label{1.2bis}
\end{eqnarray}  
\end{lemma}

\begin{proof}
We prove \eqref{1.2} by induction
on $i$. If $i=0,1$, then it is true by the definition of $\mathcal{C}_{0,0}(a)$ and \eqref{1.1}, respectively. So, assume that this formula is true for some $i-1\geq 0$, i.e. $x^{i-1}a=\sum_{k=0}^{i-1}\mathcal{C}_{k,i-1-k}(a)x^{i-1-k}$. Then, by Remark \ref{remark 2.9} we obtain that
\begin{align*}
x^{i}a &= x (x^{i-1} a) \\
&= x\cdot \left(\ \sum_{k=0}^{i-1}\mathcal{C}_{k,i-1-k}(a)x^{i-1-k}\ \right) \\
&=x \cdot \mathcal{C}_{0,i-1}(a)x^{i-1}+x\cdot \mathcal{C}_{1,i-2}(a)x^{i-2}+x\cdot \mathcal{C}_{2,i-3}(a)x^{i-3}+...+x\cdot \mathcal{C}_{i-1,0}(a)\\
&= \sigma (\mathcal{C}_{0,i-1}(a))x^{i}+\left[\delta( \mathcal{C}_{0,i-1}(a))+\sigma (\mathcal{C}_{1,i-2}(a))\right]x^{i-1}+....+\delta(\mathcal{C}_{i-1,0}(a))\\
&=\mathcal{C}_{0,i}(a)x^{i}+\mathcal{C}_{1,i-1}(a)x^{i-1}+\mathcal{C}_{2,i-2}(a)x^{i-2}+...+\mathcal{C}_{i,0}(a)\\
&=\sum_{k=0}^{i}\mathcal{C}_{k,i-k}(a)x^{i-k} \ .
\end{align*}
By reasoning in a similar way, we obtain the expression for $ax^i$ as in the statement.
\end{proof} 
 
\medskip 

From Lemma \ref{LemaCij}, it follows that
$$ax^{i}\cdot bx^{j}=\displaystyle\sum_{k=0}^{i}a \cdot \mathcal{C}_{k,i-k}(b)x^{i+j-k} \ \ , \quad 
x^i\beta\cdot x^j\alpha=\sum_{k=0}^{j}x^{i+j-k}(-1)^{k}\mathcal{T}_{k,j-k}(\beta)\alpha \ .
$$
Furthermore, given non-zero skew polynomials 
$$f(x)=\displaystyle\sum_{i=0}^{m}a_ix^{i}\ , \ g(x)=\displaystyle\sum_{j=0}^{n}b_jx^{j}\ , \ F(x)=\displaystyle\sum_{i=0}^{m}x^{i}\beta_i\ , \ G(x)=\displaystyle\sum_{j=0}^{n}x^{j}\alpha_j$$
in $\mathcal{R}$, we have
\begin{eqnarray}\label{For1.3}
f(x)g(x)=\sum_{i=0}^{m} \sum_{j=0}^{n}\left( \sum_{k=0}^{i}a_i\;  \mathcal{C}_{k,i-k}(b_j)x^{i+j-k}\right) \ ,
\end{eqnarray}
\begin{eqnarray}\label{For1.3bis}
F(x)G(x)=\displaystyle{\sum_{i=0}^{m}\sum_{j=0}^{n}\left(\sum_{k=0}^{j}x^{i+j-k}(-1)^k \mathcal{T}_{k,j-k}(\beta_i)\cdot \alpha_j\right)} \ .
\end{eqnarray}
In particular, if $\delta=0$, then we get
\begin{eqnarray}\label{For4}
\mathcal{C}_{d,s}(a)=\left \{ \begin{array}{lcc}
   a  &  \text{if} & d=s=0\\
   \sigma^{s}(a) &  \text{if} & d=0\;\text{and}\;s\neq 0 \\
   0 &  \text{if} &  d \neq 0\\
\end{array} \right.
\end{eqnarray}
\begin{center}
   $\mathcal{T}_{d,s}(a)=\left \{ 
   \begin{array}{lcc}
   a  &  \text{if} & d=s=0\\
   \sigma^{-s}(a) &  \text{if} & d=0\;\text{and}\;s\neq 0 \\
   0 &  \text{if} & d \neq 0\\
   \end{array} \right.$ 
\end{center}
for all $a\in \mathbb{F}$ and $d,s\in \mathbb{Z}_{\geq 0}$. Thus, the products $f(x)g(x)$ and $F(x)G(x)$ become simply
$$f(x)g(x)=\left(\sum_{i=0}^{m}a_ix^{i} \right)\cdot \left(\sum_{j=0}^{n}b_jx^{j}\right)=\sum_{i=0}^{m} \sum_{j=0}^{n}a_i\sigma^{i}(b_j)x^{i+j} \ ,$$
$$F(x)G(x)=\left(\sum_{i=0}^{m}x^{i}\beta_i \right)\cdot \left(\sum_{j=0}^{n}x^{j}\alpha_j\right)=\sum_{i=0}^{m} \sum_{j=0}^{n}x^{i+j}\sigma^{-j}(\beta_i)\alpha_j \ .$$

\bigskip

\noindent The following algorithms show how to compute $\mathcal{C}_{d,s}(a)$ and $\mathcal{T}_{d,s}(a)$ (see Definition \ref{Def Cij}).

\begin{algorithm}[h!]
 \small
\caption{\label{algoritmo 1} Computation of $\mathcal{C}_{d,s}(a)$.}
\label{Alg1}
\begin{algorithmic}[1]
\Require{ $\mathcal{R}$, $d,s\in\mathbb{Z}_{\geq 0}$ and $a\in\mathbb{F}$}
\Ensure{ $\mathcal{C}_{d,s}(a)$}
\State{Let $S_d\subset \mathbb{F}_2^{d+s}$ be the set of codewords with weight equal to $d$.}
\State{$\mathcal{C}_{d,s}(a) \gets 0$}
\ForAll{$(s_{1}, s_{2},\dots,s_{d+s})\in S_d$}
\State{$b\gets a$}
\For{$i\gets 1$ to $d+s$}
\If{$s_i=0$}
\State{$b\gets \sigma(b)$}
%
%
%
\Else
\State{$b\gets \delta(b)$}
\EndIf
\EndFor
\State{$\mathcal{C}_{d,s}(a)\gets \mathcal{C}_{d,s}(a)+b$}
\EndFor\\
\Return{$\mathcal{C}_{d,s}(a)$}
\end{algorithmic}
\end{algorithm}

\newpage

\begin{algorithm}[h!]
 \small
\caption{Computation of $\mathcal{T}_{d,s}(a)$.}
\label{Alg2bis}
\begin{algorithmic}[1]
\Require{$\mathcal{R}$, $\sigma^{-1}$, $d, s\in\mathbb{Z}_{\geq 0}$ and $a\in\mathbb{F}$}
\Ensure{ $\mathcal{T}_{d,s}(a)$}
\State{Let $S_d\subset \mathbb{F}_2^{d+s}$ be the set of codewords with weight equal to $d$.}
\State{$\mathcal{T}_{d,s}(a) \gets 0$}
\ForAll{$(s_{1}, s_{2},\dots,s_{d+s})\in S_d$}
\State{$b\gets a$}
\For{$i\gets 1$ to $d+s$}
\If{$s_i=1$}
\State{$b\gets \delta (\sigma^{-1}(b))$}
\Else
\State{$b\gets \sigma^{-1}(b)$}
\EndIf
\EndFor
\State{$\mathcal{T}_{d,s}(a)\gets \mathcal{T}_{d,s}(a)+b$}
\EndFor\\
\Return{$\mathcal{T}_{d,s}(a)$}
\end{algorithmic}
\end{algorithm}

\medskip

For instance, let $\mathbb{F}_4[x;\sigma,\delta]$ be the skew polynomial ring over the finite field $\mathbb{F}_4=\{ 0,1,w, w^2\}$, with $\sigma(a)=a^2$ and $\delta(a)=w(\sigma(a)+a)$ for all $a \in \mathbb{F}_4$.  As an application of Algorithm \ref{Alg1}, let us give here for this situation the following Magma program \cite{Magma}.

\medskip

\begin{verbatim}
F<w>:=GF(4); 
S:= map< F -> F | x :-> x^2 >;
D:= map< F -> F | x :-> w*(S(x)+x) >;
\end{verbatim}

\smallskip

Then, by the following instructions, we define the function ``PosCom".

\begin{prog}\label{prog 1} ~
\begin{verbatim}
PosCom:=function(d,s,a)
C:= [u: u in [VectorSpace(GF(2),d+s)!v : v in VectorSpace(GF(2),d+s)\
]| Weight(u) eq d];
A:=0;
 for k in [1..#C] do
  b:=a;
   for l in [1..d+s] do
    if C[k][l] eq 0 then
     b:=S(b);
     else
     b:=D(b); 
    end if;
   end for;
  A:=A+b;
 end for;
return A;
end function;
\end{verbatim} 
\end{prog}

\medskip

Thus, to calculate the value of $\mathcal{C}_{1,2}(w)=\delta\sigma^2(w)+\sigma\delta\sigma(w)+\sigma^2\delta(w)$,
we can simply write

\begin{verbatim} 
PosCom(1,2,w);
\end{verbatim}

\noindent which gives the answer $w^2$. 
Similarly, one can write a Magma program to calculate $\mathcal{T}_{d,s}(a)$ for any $d,s\in\mathbb{Z}_{\geq 0}$ and $a\in\mathbb{F}$.

\medskip

Using Algorithms \ref{Alg1} and \ref{Alg2bis}, one can compute the products $f\cdot g\ ,\ F\cdot G$ as in formulas \eqref{For1.3} and \eqref{For1.3bis} by the next two algorithms. 

\begin{algorithm}[h!]{}
\small
\caption{\label{Alg2} Computation of $f(x)\cdot g(x)$, where $f(x)=a_0+a_1x+\cdots+a_mx^m$ and $g(x)=b_0+b_1x+\cdots+b_nx^n$.}
\begin{algorithmic}[1]
\Require{$f,g\in\mathcal{R}$}
\Ensure{ $M=f(x)\cdot g(x)$}
\State{$M\gets 0$}
\For{$i\gets 1$ to Degree$(f)+1$}
\For{$j\gets 1$ to Degree$(g)+1$}
\For{$k\gets 0$ to $i-1$}
\State{$n\gets (i-1)+(j-1)-k$}
\State{$M\gets M+a_{i-1}\cdot \mathcal{C}_{k,i-1-k}(b_{j-1})\cdot x^n$}
\EndFor
\EndFor
\EndFor\\
\Return{$M$}
\end{algorithmic}
\end{algorithm}

\begin{algorithm}[h!]{}
\small
\caption{\label{Alg4bis} Computation of $F(x)\cdot G(x)$, where $F(x)=\beta_0+x\beta_1+\cdots+x^m\beta_m$ and $G(x)=\alpha_0+x\alpha_1+\cdots+x^n\alpha_n$.}
\begin{algorithmic}[1]
\Require{$F,G\in\mathcal{R}$}
\Ensure{ $M=F(x)\cdot G(x)$}
\State{$M\gets 0$}
\For{$i\gets 1$ to Degree$(F)+1$}
\For{$j\gets 1$ to Degree$(G)+1$}
\For{$k\gets 0$ to $j-1$}
\State{$n\gets (i-1)+(j-1)-k$}
\State{$M\gets M+x^n\cdot (-1)^k \mathcal{T}_{k,j-1-k}(b_{i-1})\cdot\alpha_{j-1}$}
\EndFor
%
%
%
%
\EndFor
\EndFor\\
\Return{$M$}
\end{algorithmic}
\end{algorithm}

\medskip

For instance, as an application of Algorithm \ref{Alg2}, we give here a Magma program to compute the products $fg$ and $gf$ when $f=x^2+1$ and $g=x^2+i$ are skew polynomials in $\mathbb{C}[x;\sigma,\delta]$ with $\sigma(z)=\bar{z}$ and $\delta(z)=z-\bar{z}$ for all $z\in \mathbb{C}$. 

\medskip

We begin by writing the following instructions:

\smallskip

\begin{verbatim}
F<i>:=ComplexField();
R<x>:=PolynomialRing(F);
S:= map< F -> F | x :-> ComplexConjugate(x) >;
D:= map< F -> F | x :-> x-ComplexConjugate(x) >;
\end{verbatim}

\smallskip

then, using the function ``PosCom" defined in Program \ref{prog 1}, we can continue with the  following instructions to define the new function ``MultPol".

\begin{prog}\label{prog 2} ~
\begin{verbatim}
MultPol:=function(f,g)
M:=0;
 for i in [1..#f] do
  for j in [1..#g] do
   for k in [0..i-1] do
    n:=(i-1)+(j-1)-k;
    M:=M+f[i]*PosCom(k,i-k-1,g[j])*x^n;
   end for;
  end for;
 end for;
return M;
end function;
\end{verbatim}
\end{prog}

\medskip

Thus, to calculate $(x^2+1)(x^2+i)$ and $(x^2+i)(x^2+1)$, we write in Magma
\begin{verbatim}
MultPol([1,0,1],[i,0,1]);
MultPol([i,0,1],[1,0,1]);
\end{verbatim}
\noindent obtaining
$x^4 + (1 +i)x^2 - 4ix + 5i$ and
$x^4 + (1 + i)x^2 + i$, respectively.

\bigskip

Finally, let us recall here also the process of ``evaluating'' a skew polynomial $f(x)\in\mathcal{R}$ at an element $a\in \mathbb{F}$. To define an evaluation map for a skew polynomial ring, we need to consider the action of $\sigma$ and $\delta$. Indeed, the classical map that simply replaces the variable $x$ with a value $a \in \mathbb{F}$ does not work in $\mathcal{R}$. So, Lam and Leroy in  \cite[p.~310]{lamandleroy} defined an appropriate evaluation using the fact that $\mathcal{R}$ is a right Euclidean domain.

\begin{definition}
For $a \in \mathbb{F}$ and $f \in\mathcal{R}$, where $\sigma$ is an endomorphism (automorphism) of $\mathbb{F}$, we define the right (left) evaluation of $f$ in $a$, denoted by $f(a)$ $(f_L(a))$, as the unique remainder upon right (left) division of
$f$ by $x-a$. In the special case when $f(a)=0$ $(f_L(a)=0)$, we say that $a$ is a \textit{right (left) zero} of $f$.
\end{definition}

We can also compute the right (left) evaluation of a polynomial in $\mathcal{R}$ at $a\in\mathbb{F}$ without using the right (left) division algorithm. To do this, we first need the following technical result (see also \cite[Theorem 7]{Ore}).

\begin{lemma}\label{Res3.1}
Let $\sigma$ be an automorphism of $\mathbb{F}$. Every skew polynomial  $f=\sum_{i=0}^{m}a_ix^{i}\in\mathcal{R}$ can be represented as a polynomial with right-hand coefficients, that is, we can write 
$$f=\sum_{i=0}^{m}a_ix^{i}=\sum_{i=0}^{m}x^{i}\mathcal{A}_i \ ,$$ 
where
\begin{equation} \label{For3.3}
\mathcal{A}_i:=\sum_{j=0}^{m-i}(-1)^{j}\mathcal{T}_{j,i}(a_{j+i})\ , \ \ \forall i=0,...,m \ .
\end{equation} 
\end{lemma}

\begin{proof}
By Lemma \ref{LemaCij} we have 
$$
\sum_{i=0}^{m}a_ix^{i}=\sum_{i=0}^{m}\left(\ \sum_{k=0}^{i} x^{i-k}(-1)^{k}\mathcal{T}_{k,i-k}(a_i)\ \right)= \left(\ \sum_{h=0}^{m} x^0(-1)^h\mathcal{T}_{h,0}(a_h)\ \right)+
$$
$$   + \left(\ \sum_{h=1}^{m} x^1(-1)^{h-1}\mathcal{T}_{h-1,1}(a_{h})\ \right)+ \dots +
    \left(\ \sum_{h=m}^{m} x^m(-1)^{h-m}\mathcal{T}_{h-m,m}(a_{h})\ \right)=
$$  
$$ = x^0\left(\ \sum_{j=0}^{m}(-1)^j\mathcal{T}_{j,0}(a_j)\ \right)+ x^1\left(\ \sum_{j=0}^{m-1} (-1)^j\mathcal{T}_{j,1}(a_{j+1})\ \right)+ \dots +
    x^m\left(\ \sum_{j=0}^{m-m} (-1)^j\mathcal{T}_{j,m}(a_{j+m})\ \right)
$$
and this leads to the statement. 
\end{proof}

\begin{lemma}\label{F}{ \normalfont \cite[Lemma 2.4]{lamandleroy}}
For $f(x)=\sum_{i}a_ix^{i}\in \mathbb{F}[x;\sigma, \delta]$ and $a\in \mathbb{F}$, we have $f(a) = \sum_{i} a_i N_i^{\sigma,\delta}(a)$, where $N_0^{\sigma,\delta}(a) :=1$ and $N_{i}^{\sigma,\delta}(a) :=\sigma(N_{i-1}^{\sigma,\delta}(a))a+\delta(N_{i-1}^{\sigma,\delta}(a))$.
\end{lemma}

\begin{lemma}\label{evizquierda}{ \normalfont \cite[Theorem 3.1]{phdtesis}}
Let $\sigma$ be an automorphism of $\mathbb{F}$. For $f(x)=\sum_{i}a_ix^{i}=\sum_{i}x^i\mathcal{A}_i\in \mathbb{F}[x;\sigma, \delta]$ and $a\in \mathbb{F}$, we have $f_L(a) = \sum_{i} M_i^{\sigma,\delta}(a)\mathcal{A}_i$, where $M_0^{\sigma,\delta}(a) :=1$ and $M_{i}^{\sigma,\delta}(a) :=a\sigma^{-1}(M_{i-1}^{\sigma,\delta}(a))-\delta\sigma^{-1}(M_{i-1}^{\sigma,\delta}(a))$.
\end{lemma}

\begin{example}
In $\mathbb{F}_4[x;\sigma,0]$ with $\sigma(a)=a^2$, the binomials $(x+1),(x+\alpha^2)$ and $(x+\alpha)$ are all right (left) factors of $f(x):=x^2+1$. Thus, $\mathcal{R}$ is not in general a unique factorization domain. Moreover, $\{1,\alpha,\alpha^2\}$ are right (left) zeros of $f(x)$, showing that in general a skew polynomial of degree $n\geq 2$ could have more than $n$ roots, possibly infinite. For instance, consider $\mathbb{C}[x;\sigma,0]$ with $\sigma$ the complex conjugation (i.e. $\sigma(w)=\bar{w}$ for all $w\in\mathbb{C}$). Then,  since $\sigma^{-1}=\sigma$, by applying Lemmas \ref{F} and \ref{evizquierda}, we see that all the complex numbers $z$ such that $|z|=1$ are right (left) roots of the polynomial $x^2-1\in\mathbb{C}[x;\sigma,0]$.
\end{example}

\medskip

For $f(x) = g(x)h(x) \in\mathcal{R}$ and $a\in \mathbb{F}$,
we do not have $f(a) = g(a)h(a)$ in general. To properly define the right (left) evaluation of a product, we first need the notion of the right (left) $(\sigma,\delta)$-conjugacy (see \cite[p.~311--312]{lamandleroy} and \cite[p. 24 for $\delta=0$]{phdtesis}).

\begin{definition}\label{conj}
Given $a \in \mathbb{F}, \; c \in \mathbb{F}^{*}:=\mathbb{F}\setminus\{0\}$, we define the right (left) \textit{$(\sigma,\delta)$-conjugate} $a^c$ (\ ${}^{c}a$\ ) of $a$ with respect to $c$ as
$$a^c := \sigma(c)ac^{-1} + \delta(c)c^{-1}\in \mathbb{F}
\qquad (\ {}^{c}a := c^{-1}a\sigma^{-1}(c)-c^{-1}\delta(\sigma^{-1}(c))\in\mathbb{F} \ ) \ . $$
\end{definition}

\begin{remark}\label{lem}
Given $a,b \in \mathbb{F}$ and $c,d \in \mathbb{F}^{*}$, we have ${}^{1}a=a$, ${}^{d}({}^{c}a)={}^{cd}a$ and the relation $\sim_L$ defined on $\mathbb{F}$ as
    
\smallskip    

\centerline{$a \sim_L b$ $\iff\ \exists \ e \in \mathbb{F}^{*}$ such that $b={}^{e}a$,} 

\smallskip    

\noindent is an equivalence relation on $\mathbb{F}$. 
\end{remark}

\noindent Using Definition \ref{conj}, the following result provides formulas for right (left) evaluating a product (see \cite[Theorem 2.7]{lamandleroy} and \cite[Theorem 3.2 for $\delta=0$]{phdtesis}).

\begin{theorem}\label{F1}
Let $f(x),g(x) \in\mathcal{R}$ and $a \in \mathbb{F}$. Then the following properties hold:
\begin{enumerate}
\item[$1)$] If $g(a) = 0$, then $(f\cdot g)(a) = 0$; if $g(a) \neq 0$ then $(f\cdot g)(a) = f\left(a^{g(a)}\right)g(a)$;
\item[$2)$] If $g_L(a) = 0$, then $(g\cdot f)_L(a) = 0$; if $g_L(a) \neq 0$ then $(g\cdot f)_L(a) = g_L(a)f_L\left({}^{g_L(a)}a\right)$.
\end{enumerate}
\end{theorem}

\begin{proof}
The statements $1)$ and $2)$ follows directly from Theorem 2.7 of \cite{lamandleroy} and a slight modification of its proof. 
\end{proof}

\section{Resultants of skew polynomials}

Unless otherwise stated, $\mathbb{F}$ will denote a division ring. Given $f,g\in \mathbb{F}[x;\sigma,0]$ non-constant skew polynomials, it is known that $f$ and $g$ have a common (non-unit) right factor in $\mathbb{F}[x;\sigma,0]$ if and only if $R(f,g)=0$, where $R(f,g)$ denotes the \textit{resultant} of $f$ and $g$ over $\mathbb{F}$ (see \cite[Theorem 2.5]{ResEric}).
The main goal of this section will be to extend this result to the case $\delta \neq 0$ and to show further equivalent conditions.

\subsection{Right $(\sigma,\delta)$-Resultant}\label{RRes}
We begin by proving two technical results which are useful to define the so called right $(\sigma,\delta)$-resultant of two skew polynomials in $\mathcal{R}$ (see Definition \ref{Res1.3}). 

\begin{lemma} \label{Res1.1}
Let $f,g\in \mathcal{R}$ be non-constant skew polynomials. The following hold:
\begin{itemize}
    \item[$1)$] $\mathcal{R}/\mathcal{R}f$ is a left $\mathbb{F}$-module and ${\dim}\;\mathcal{R}/\mathcal{R}f={\deg(f)}$;
    \item[$2)$] if $\mathcal{R}g\subseteq\mathcal{R}f$, then $\mathcal{R}f/\mathcal{R}g$ is a left $\mathbb{F}$-module and ${\dim}\;\mathcal{R}f/\mathcal{R}g={\deg(g)-\deg(f)}$;
    \item[$3)$] if $k, h \in \mathcal{R}$ are such that $\mathcal{R}f \cap \mathcal{R}g = \mathcal{R}h$ and $\mathcal{R}f + \mathcal{R}g = \mathcal{R}k,$ then $$\deg(f) + \deg(g) = \deg(h) + \deg(k)\ .$$
\end{itemize}
\end{lemma}

\begin{proof} 
$1)$ Defining in $\mathcal{R}/\mathcal{R}f:=\{p+\mathcal{R}f:p\in \mathcal{R}\}$ the usual operations of addition and scalar multiplication (on the left) given by 
    $ (p_1+\mathcal{R}f) + (p_2+\mathcal{R}f): = (p_1 + p_2)+\mathcal{R}f $ and $ \alpha (p_1+\mathcal{R}f): = \alpha p_1+\mathcal{R}f $, for all $p_1,p_2\in \mathcal{R}$ and $\alpha \in \mathbb{F}$, one can see easily that $\mathcal{R}/\mathcal{R}f$ is a left $\mathbb{F}$-module. Since every coset in $\mathcal{R}/\mathcal{R}f$ contains a unique representative of degree less than $\deg(f)$, it follows that $B:=\{1+\mathcal{R}f, x+\mathcal{R}f, x^2+\mathcal{R}f, ..., x^{\deg(f)-1}+\mathcal{R}f \}$ is a left basis for $\mathcal{R}/\mathcal{R}f$. Therefore, $\dim \mathcal{R}/\mathcal{R}f=\deg(f)$.

$2)$ Since $g\in \mathcal{R}g\subseteq \mathcal{R}f$, we have $g=hf$ for some $h\in \mathcal{R}$. Thus, we can write $Rf/Rg=\{rf+Rhf:r\in R\}$. On the other hand, since $\psi:\mathcal{R}\to \mathcal{R}f/\mathcal{R}hf$, $p \mapsto pf+\mathcal{R}hf$ is a surjective left $\mathbb{F}$-module homomorphism  with $\ker\psi=\mathcal{R}h$, we have $\mathcal{R}/\mathcal{R}h \cong \mathcal{R}f/\mathcal{R}hf=\mathcal{R}f/\mathcal{R}g$. Finally, by $1)$ it follows that $\dim \mathcal{R}f/\mathcal{R}g=\dim \mathcal{R}/\mathcal{R}h=\deg(h)=\deg(g)-\deg(f)$

$3)$ Since $\mathcal{R}$ is a LPID, we can write $\mathcal{R}f \cap \mathcal{R}g = \mathcal{R}h$ and $\mathcal{R}f + \mathcal{R}g = \mathcal{R}k$ for some $h,k\in \mathcal{R}$. Since $Rf,Rg,Rh$ and $Rk$ are left $\mathbb{F}$-submodules of $\mathcal{R}$, we deduce that
    $(Rf+Rg)/Rf\cong Rg/(Rf \cap Rg)$, i.e. $ Rk/Rf \cong Rg/Rh$.
Hence $\dim Rk/Rf=\dim Rg/Rh $ and by $2)$ we have $\deg(f) + \deg(g) = \deg(h) + \deg(k)$.
\end{proof}

The previous lemma is an extension of some results showed in \cite[p.~4]{ResEric}. We have given here its proof only for convenience of the reader. Moreover, by Lemma \ref{Res1.1} it is possible to prove also the following technical result, but we omit its proof because it is analogous to the one presented in \cite[Theorem 2.4]{ResEric} for the case $\delta=0$.
 
\begin{lemma}\label{Res1.2}
 Two non-constant skew polynomials $f,g\in \mathcal{R}$ of respective degrees $m$ and $n$, have a common (non-unit) right factor in $\mathcal{R}$, if and only if there exist skew polynomials $c,d \in \mathcal{R}$ such that $cf+dg=0$, $\deg(c) < n$ and $\deg(d)< m$. 
\end{lemma}

By Lemma \ref{Res1.2} we can define the right $(\sigma,\delta)$-resultant of two skew polynomials in $\mathcal{R}$ as shown below. Let 
\begin{align*}
f &= a_mx^m+...+a_1x+a_0, a_m \neq 0 \;,\;\;\;\; g=b_nx^n+...+b_1x+b_0, b_n \neq 0\;,\\
 c &= c_{n-1}x^{n-1}+...+c_1x+c_0 \;,\;\;\;\;\;\;\;\;\;\;\;\;d= d_{m-1}x^{m-1}+...+d_1x+d_0
\end{align*}
 be skew  polynomials as in Lemma \ref{Res1.2}. By \eqref{For1.3}, we have
$$cf= \sum_{i=0}^{n-1}\sum_{j=0}^{m}\left(\sum_{k=0}^{i} c_i\cdot \mathcal{C}_{k,i-k}(a_j)x^{i+j-k} \right), \; dg=\sum_{i=0}^{m-1}\sum_{j=0}^{n}\left(\sum_{k=0}^{i} d_i\cdot \mathcal{C}_{k,i-k}(b_j)x^{i+j-k} \right)$$
Keeping in mind that two skew polynomials are equal if and only if they have the same degree and their respective coefficients are equal, the equation $cf+dg=0$ of Lemma \ref{Res1.2} gives a system of $m + n$ linear equations with $m + n$ unknowns $c_0,..., c_{n-1}, d_0,..., d_{m-1}$, that is \begin{equation} \label{For3.1}
(c_0, ..., c_{n-1}, d_0, ..., d_{m-1})\cdot A=(0,...,0)\;,
\end{equation}
where $A$ is the following $(m+n)\times(m+n)$ matrix:

\begin{center}
\resizebox{16.5 cm}{4 cm}{
\begin{tabular}{c} 
$A:=\begin{pmatrix}
a_0& a_1 & a_2 & \cdots & a_m & 0 & 0 & 0& \cdots  & 0 \\ 
\mathcal{C}_{1,0}(a_0)& \displaystyle\sum_{i=0}^{1}\mathcal{C}_{1-i,i}(a_{1-i}) &  \displaystyle\sum_{i=0}^{1}\mathcal{C}_{1-i,i}(a_{2-i}) & \cdots & \displaystyle\sum_{i=0}^{1}\mathcal{C}_{1-i,i}(a_{m-i}) & \mathcal{C}_{0,1}(a_m)  & 0& 0 &\cdots & 0 \\ 
\mathcal{C}_{2,0}(a_0) &  \displaystyle\sum_{i=0}^{1}\mathcal{C}_{2-i,i}(a_{1-i}) & \displaystyle\sum_{i=0}^{2}\mathcal{C}_{2-i,i}(a_{2-i}) & \cdots & \displaystyle\sum_{i=0}^{2}\mathcal{C}_{2-i,i}(a_{m-i})& \displaystyle\sum_{i=1}^{2}\mathcal{C}_{2-i,i}(a_{m+1-i}) & \mathcal{C}_{0,2}(a_m) & 0 & \cdots &0 \\ 
\mathcal{C}_{3,0}(a_0) &  \displaystyle\sum_{i=0}^{1}\mathcal{C}_{3-i,i}(a_{1-i}) & \displaystyle\sum_{i=0}^{2}\mathcal{C}_{3-i,i}(a_{2-i}) & \cdots & \displaystyle\sum_{i=0}^{3}\mathcal{C}_{3-i,i}(a_{m-i})& \displaystyle\sum_{i=1}^{3}\mathcal{C}_{3-i,i}(a_{m+1-i}) & \displaystyle\sum_{i=2}^{3}\mathcal{C}_{3-i,i}(a_{m+2-i}) &  \mathcal{C}_{0,3}(a_m) & \cdots &0 \\ 
\vdots & \vdots & \vdots & \cdots  & \vdots & \vdots  & \vdots & \ddots  & \ddots& \vdots\\ 
\mathcal{C}_{n-1,0}(a_0) & \displaystyle\sum_{i=0}^{1}\mathcal{C}_{n-1-i,i}(a_{1-i}) & \displaystyle\sum_{i=0}^{2}\mathcal{C}_{n-1-i,i}(a_{2-i}) & \cdots & \displaystyle\sum_{i=0}^{m}\mathcal{C}_{n-1-i,i}(a_{m-i})& \displaystyle\sum_{i=1}^{n-1}\mathcal{C}_{n-1-i,i}(a_{m+1-i}) & \displaystyle\sum_{i=2}^{n-1}\mathcal{C}_{n-1-i,i}(a_{m+2-i}) & \displaystyle\sum_{i=3}^{n-1}\mathcal{C}_{n-1-i,i}(a_{m+3-i}) & \cdots & \mathcal{C}_{0,n-1}(a_m) \\
b_0& b_1 & b_2 & \cdots & b_n & 0 & 0 & 0& \cdots  & 0 \\ 
\mathcal{C}_{1,0}(b_0)& \displaystyle\sum_{i=0}^{1}\mathcal{C}_{1-i,i}(b_{1-i}) &  \displaystyle\sum_{i=0}^{1}\mathcal{C}_{1-i,i}(b_{2-i}) & \cdots & \displaystyle\sum_{i=0}^{1}\mathcal{C}_{1-i,i}(b_{n-i}) & \mathcal{C}_{0,1}(b_n)  & 0& 0 &\cdots & 0 \\ 
\mathcal{C}_{2,0}(b_0) &  \displaystyle\sum_{i=0}^{1}\mathcal{C}_{2-i,i}(b_{1-i}) & \displaystyle\sum_{i=0}^{2}\mathcal{C}_{2-i,i}(b_{2-i}) & \cdots & \displaystyle\sum_{i=0}^{2}\mathcal{C}_{2-i,i}(b_{n-i})& \displaystyle\sum_{i=1}^{2}\mathcal{C}_{2-i,i}(b_{n+1-i}) & \mathcal{C}_{0,2}(b_n) & 0 & \cdots &0 \\ \mathcal{C}_{3,0}(b_0) &  \displaystyle\sum_{i=0}^{1}\mathcal{C}_{3-i,i}(b_{1-i}) & \displaystyle\sum_{i=0}^{2}\mathcal{C}_{3-i,i}(b_{2-i}) & \cdots & \displaystyle\sum_{i=0}^{3}\mathcal{C}_{3-i,i}(b_{n-i})& \displaystyle\sum_{i=1}^{3}\mathcal{C}_{3-i,i}(b_{n+1-i}) & \displaystyle\sum_{i=2}^{3}\mathcal{C}_{3-i,i}(b_{n+2-i}) &  \mathcal{C}_{0,3}(b_n) & \cdots &0 \\ 
\vdots & \vdots & \vdots & \cdots  & \vdots & \vdots  & \vdots & \ddots  & \ddots& \vdots\\ 
\mathcal{C}_{m-1,0}(b_0) & \displaystyle\sum_{i=0}^{1}\mathcal{C}_{m-1-i,i}(b_{1-i}) & \displaystyle\sum_{i=0}^{2}\mathcal{C}_{m-1-i,i}(b_{2-i}) & \cdots & \displaystyle\sum_{i=0}^{n}\mathcal{C}_{m-1-i,i}(b_{n-i})& \displaystyle\sum_{i=1}^{m-1}\mathcal{C}_{m-1-i,i}(b_{n+1-i}) & \displaystyle\sum_{i=2}^{m-1}\mathcal{C}_{m-1-i,i}(b_{n+2-i}) & \displaystyle\sum_{i=3}^{m-1}\mathcal{C}_{m-1-i,i}(b_{n+3-i}) & \cdots & \mathcal{C}_{0,m-1}(b_n)
\end{pmatrix}$ 
\end{tabular}
}
\end{center} 

\noindent Note that the first $n$ rows involve the $a_i$'s and the last $m$ rows involve the $b_j$'s. \medskip

By the previous $(m+n)\times(m+n)$ matrix $A$, we can define the right $(\sigma,\delta)$-resultant of two skew polynomials in $\mathcal{R}$ as follows.

\begin{definition}\label{Res1.3}
Let $f,g\in\mathcal{R}$ be skew polynomials of non-negative degrees $m$ and $n$, respectively. The above matrix $A$ will be called the \textit{right $(\sigma,\delta)$-Sylvester matrix} of $f$ and $g$, which we denote by $\text{Sylv}_{\mathbb{F}}^{\sigma,\delta}(f,g)$. We define the \textit{right $(\sigma,\delta)$-\textit{resultant}} of $f$ and $g$ (over $\mathbb{F}$), denoted by $R^{\sigma,\delta}_{\mathbb{F}}(f,g)$, as the Dieudonn\'e determinant of $\text{Sylv}_{\mathbb{F}}^{\sigma,\delta}(f,g)$.
\end{definition}

Let us recall that the Dieudonn\'e determinant, denoted here by Ddet, is a non-commutative generalization of the classical determinant of a matrix with entries in a field, to matrices over division rings. This determinant takes
values in $\{0\}\cup\mathbb{F}^{*}/[\mathbb{F}^{*},\mathbb{F}^{*}]$, where $[\mathbb{F}^{*},\mathbb{F}^{*}]$ is the commutator of
the multiplicative group $\mathbb{F}^{*}:=\mathbb{F}\setminus \{0\}$. If $\mathbb{F}$ is a field, then Ddet coincides with the classical definition of determinant and in this case, we will write simply $\det$ instead of Ddet. For more details on the properties of
the Dieudonn\'e determinant, see e.g \cite{dieudone},  \cite[p.~151]{artin} and \cite[p.~133]{draxl}.

\medskip

\begin{remark}
In the special case when $\delta=0$, $R^{\sigma, \delta}_{\mathbb{F}}(f,g)$ coincides with the resultant $R(f,g)$ defined in \cite[p.~6]{ResEric}. In fact, by \eqref{For4} we have 
$$R^{\sigma,0}_{\mathbb{F}}(f,g)=\text{Ddet} \begin{pmatrix}
a_{0} & a_{1} & a_2 & \cdots  & a_m & 0 & \cdots & 0\\
0 & \sigma(a_{0}) & \sigma(a_{1})  & \cdots & \sigma(a_{m-1}) & \sigma(a_m) & \cdots & 0\\
\vdots & \vdots & \ddots  & \vdots & \vdots & \vdots & \ddots & \vdots &\\
0 & 0 & 0 & \cdots & \sigma^{n-1}(a_0) & \sigma^{n-1}(a_{1}) & \cdots & \sigma^{n-1}(a_m) \\
b_{0} & b_{1} & b_2 & \cdots  & b_n & 0 & \cdots & 0\\
0 & \sigma(b_{0}) & \sigma(b_{1})  & \cdots & \sigma(b_{n-1}) & \sigma(b_n) & \cdots & 0\\
\vdots & \vdots & \ddots  & \vdots & \vdots & \vdots & \ddots & \vdots &\\
0 & 0 & 0 & \cdots & \sigma^{m-1}(b_0) & \sigma^{m-1}(b_{1}) & \cdots & \sigma^{m-1}(b_n) 
\end{pmatrix}$$
Furthermore, if $\sigma={Id}$ then we obtain the classical notion of resultant.
\end{remark}
 
Applying Algorithm \ref{Alg1}, the next algorithm shows how to find the right $(\sigma,\delta)$-Silvester matrix of $f$ and $g$ (see Definition \ref{Res1.3}). 

\begin{algorithm}[h!]
\small
\caption{\label{algoritmo 3} Computation of the right $(\sigma,\delta)$-Sylvester matrix of $f(x)=a_0+a_1x+\cdots+a_mx^m$ and $g(x)=b_0+b_1x+\cdots+b_nx^n$.}
\label{Alg3}
\begin{algorithmic}[1]
\Require{$f,g\in\mathcal{R}$.}
\Ensure{$(\sigma,\delta)$-Sylvester matrix $M$ of $f$ and $g$.}
\State{$M_1\gets \begin{pmatrix}  a_{0}  &   a_{1} &a_2 & \cdots & a_{n+m} \end{pmatrix}$ }
\State{$M_2\gets \begin{pmatrix}  b_{0}  &   b_{1}&b_2 & \cdots & b_{n+m} \end{pmatrix}$ }
\For{$p\gets 1$ to $n-1$}
\State{$M_3\gets \begin{pmatrix}  \mathcal{C}_{p,0}(a_{0}) \end{pmatrix}$}
\For{$q\gets 1$ to $n+m-1$}
\State{$Z_1\gets 0$}
\For{$l\gets 0$ to $p$}
\If{$0\leq q-l\leq m$}
\State{$Z_1 \gets Z_1 + \mathcal{C}_{p-l,l}(a_{q-l})$}
\EndIf
 \EndFor
\State{$M_3\gets \left( \begin{array}{c|c} M_3 & Z_1 \end{array}\right)$}
\EndFor
\State{$M_1\gets \left( \begin{array}{c} M_1 \\ \hline    M_3 \end{array}\right)$}
\EndFor 
\For{$p\gets 1$ to $m-1$}
\State{$M_4\gets \begin{pmatrix}  \mathcal{C}_{p,0}(b_{0}) \end{pmatrix}$}
\For{$q\gets 1$ to $n+m-1$}
\State{$Z_2\gets 0$}

\algstore{myalg}
\end{algorithmic}
\end{algorithm}

\begin{algorithm}[h!]                     
\begin{algorithmic} [1]                  
\algrestore{myalg}

\For{$l\gets 0$ to $p$}
\If{$0\leq q-l\leq n$}
\State{$Z_2 \gets Z_2 + \mathcal{C}_{p-l,l}(b_{q-l})$}
\EndIf
\EndFor
\State{$M_4\gets \left( \begin{array}{c|c} M_4 & Z_2 \end{array}\right)$}
\EndFor
\State{$M_2\gets \left( \begin{array}{c} M_2\\ \hline   M_4 \end{array}\right)$}
\EndFor \\ 
\Return{$M\gets \left( \begin{array}{c} M_1 \\ \hline   M_2 \end{array}\right)$}
\end{algorithmic}
\end{algorithm}

\bigskip

As an application of Algorithm \ref{Alg3}, let us give here the following Magma program to compute $\text{Sylv}_{\mathbb{H}}^{\sigma,\delta}(f,g)$ when $f=x^4+kx^3-jx^2-i$ and $g=x^3+j$ are skew polynomials in $\mathbb{H}[x;\sigma,0]$ with $\sigma(h):=ihi^{-1}$ for all $h\in \mathbb{H}$.
\medskip
\begin{verbatim}
F<i,j,k> := QuaternionAlgebra< RealField() | -1, -1 >;
S:= map< F -> F | x :-> i*x*(1/i) >;
D:= map< F -> F | x :-> 0 >;
\end{verbatim}

\noindent Then, using the function ``PosCom" defined in Program \ref{prog 1}, we can define the new function ``SylvesterMatrix" with previously the function ``SumPosCom" as follows.
\begin{prog} \label{Program:Sylvester} ~
\begin{verbatim}
SumPosCom:=function(f,i,j)
AA:=0;
n:=#f-1;
 for I in [0..i-1] do
  if j-1-I ge 0 and j-1-I le n then
   if i-1 ne 0 then
    AA:=AA+PosCom(i-1-I,I,f[j-I]);
    else
    AA:=f[j-I];
   end if;
  end if;
 end for;
return AA;
end function;

SylvesterMatrix:=function(f,g)
n:=#f-1;
m:=#g-1;
if m ne 0 then
 M1:= Matrix(F,1,n+m,[SumPosCom(f,s,t): s in {1}, t in {1..n+m}]);
  for p in [2..m] do
   X:=Matrix(F,1,n+m,[SumPosCom(f,s,t): s in {p}, t in {1..n+m}]);
   M1:=VerticalJoin(M1,X);
  end for;
else
 M1:=RemoveRow(ZeroMatrix(F,1,n+m),1);
end if;
if n ne 0 then
 M2:= Matrix(F,1,n+m,[SumPosCom(g,s,t): s in {1}, t in {1..n+m}]);
  for p in [2..n] do
   X:=Matrix(F,1,n+m,[SumPosCom(g,s,t): s in {p}, t in {1..n+m}]);
   M2:=VerticalJoin(M2,X);
  end for;
else
 M2:=RemoveRow(ZeroMatrix(F,1,n+m),1);
end if;
M:=VerticalJoin(M1,M2);
return M;
end function;
\end{verbatim}
\end{prog}
Then, by typing in Magma
\begin{verbatim}
SylvesterMatrix([-i,-j,0,k,1],[j,0,0,1]);
\end{verbatim}
\noindent we obtain the right $(\sigma,\delta)$-Sylvester matrix of $f(x)=x^4+kx^3-jx-i$ and $g(x)=x^3+j$:
\begin{equation} \label{M}
{\tiny
\left(
\begin{matrix}
-i & -j & 0 & k & 1 & 0 & 0 \\
 0 & -i & j & 0 & -k & 1 & 0 \\
 0 & 0 & -i & -j & 0 & k & 1 \\
 j & 0 & 0 & 1 & 0 & 0 & 0 \\
 0 & -j & 0 & 0 & 1 & 0 & 0 \\ 
 0 & 0 & j & 0 & 0 & 1 & 0 \\ 
 0 & 0 & 0 & -j & 0 & 0 & 1
\end{matrix}\right)\ .
}
\end{equation}

\begin{remark}
When $\mathbb{F}$ is a field, we can write $\text{R}_{\mathbb{F}}^{\sigma,\delta}(f,g):=\det(\text{Sylv}_{\mathbb{F}}^{\sigma,\delta}(f,g))$, where $\det$ is the classical determinant. Therefore, in this case, we can easily compute $\text{R}_{\mathbb{F}}^{\sigma,\delta}(f,g)$ in Magma by using the command ``Determinant(~)''. However, this command in Magma generates difficulties in some situations. For example, when $\mathbb{F}$ is the field of the complex numbers, this field can only be dealt with a certain level of precision, and therefore  Magma cannot give the exact value of the determinant. For this reason, we provide below a Magma program (Program \ref{Program:Sarrus}) based on 
the definition of the determinant $\det A$ of an $n\times n$ matrix $A$ with entries $a_{ij}\in\mathbb{F}$ using the Leibniz's formula, i.e.
$$\det A:=\sum_{\Sigma\in S_n}\left(\mathrm{sgn}(\Sigma)a_{1,\Sigma_1}\dots a_{n,\Sigma_n} \right) \ ,$$ 
where $S_n$ is the symmetric group of $n$ elements, sgn$(\Sigma)$ is the sign of the permutation $\Sigma\in S_n$ and $\Sigma_i$ is the value in the $i$-th position after the reordering $\Sigma$.  
The advantage of this Magma program is that it avoids the Gaussian elimination and consequently the computation of quotients, because it only works with sums and products. 

\newpage

\begin{prog} \label{Program:Sarrus} ~
\begin{verbatim}
Det:=function(M)
n:=NumberOfColumns(M);
P:=[ p : p in Permutations({a : a in [1..n]})];
S2:=0;
 for k in [1..#P] do
  S1:=1;
  for j in [1..n] do
   S1:=S1*M[j,P[k][j]];
  end for;
  g:=Sym(n)!P[k];
  if IsEven(g) then
   S2:=S2+S1;
   else 
   S2:=S2-S1;
  end if;
 end for;
return S2;
end function;
\end{verbatim}
\end{prog}
\end{remark}

\bigskip

Now, let us give here the main results of this section for polynomials in $\mathcal{R}$.

\begin{theorem}\label{Res1.5}
Let $f,g\in \mathcal{R}$ be non-constant skew polynomials of degrees $m$ and $n$, respectively. The following conditions are equivalent:
\begin{itemize}
    \item[$1)$]$\text{R}^{\sigma,\delta}_{\mathbb{F}}(f,g)=0$;
    \item[$2)$] $f$ and $g$ have a common (non-unit) right factor in $\mathcal{R}$;
    \item[$3)$] $\text{gcrd}(f,g)\neq 1$ (where "gcrd" means greatest  common  right  divisor);
    \item[$4)$] there are no polynomials $p,q \in \mathcal{R}$ such that $pf + qg = 1$;
    \item[$5)$] $\mathcal{R}f+\mathcal{R}g\subsetneq \mathcal{R}$.
\end{itemize}
\end{theorem}
\begin{proof}
$1)\Leftrightarrow 2):$ $R^{\sigma,\delta}_{\mathbb{F}}(f,g):=\text{Ddet}(A)=0$ if and only if the homogeneous linear system \eqref{For3.1} has a non-trivial solution. The above, is equivalent to say that there exist
skew polynomials $c,d\in \mathcal{R}$ such that $cf+dg=0$, $\deg(c)< n$ and $\deg(d)<m$. However, by Lemma \ref{Res1.2}, the latter is true if and only if $f$ and $g$ have a common (non-unit) right factor in $\mathcal{R}$.
  
$2)\Leftrightarrow 3):$ obvious.

$3) \Rightarrow 4):$ Let $r\in \mathcal{R}$ be the common (non-unit) right factor of $f$ and $g$ (it exists because $\text{gcrd}(f,g)\neq 1$). Then, $f=q_1r$ and $g=q_2r$, for some $q_1,q_2\in \mathcal{R}$. Since for all $p,q\in \mathcal{R}$, $pf+qg=(pq_1+qq_2)r$, it follows that $pf+qg\neq 1$.

$4) \Rightarrow 3):$ Assume that for all $p,q\in \mathcal{R}$, $pf+qg\neq 1$. Since $\mathcal{R}$ is a LPID, we can write $\mathcal{R}f+\mathcal{R}g=\mathcal{R}h \varsubsetneq \mathcal{R}$, for some $h\in \mathcal{R}$ of positive degree. Thus $h$ is a common (non-unit) right factor of $f$ and $g$. 

$4) \Leftrightarrow 5):$  It follows from the fact that $\mathcal{R}f+\mathcal{R}g=\mathcal{R}$ if and only if $1\in \mathcal{R}f+\mathcal{R}g$. 
\end{proof}

\smallskip

\begin{remark}
When $\delta=0$, the equivalence between $1)$ and $2)$ in Theorem \ref{Res1.5} gives Theorem 2.5 in \cite{ResEric}. Moreover, if $\mathbb{F}=\mathbb{H}$ (Hamilton's quaternions), $\sigma={Id}$ and $\delta=0$, then the equivalence between $1)$ and $3)$ in Theorem \ref{Res1.5} gives also an analogous result to \cite[Theorem 4.3]{xianguiresultants}, but with a different notion of determinant.
\end{remark}

\medskip

In what follows, the objective is to determine if the Dieudonn\'e determinant of any matrix is zero or not in line with $1)$ of Theorem \ref{Res1.5}. 
To do this, we first need Algorithm \ref{Alg4} to obtain (via elementary row operations on the left) the corresponding upper triangular matrix $D$  of any matrix $M$ with entries in~$\mathbb{F}$. Note that this operation does not change the nullity of $M$.

\begin{algorithm}[h!]
\small
\caption{\label{algoritmo 4} Computation of the upper triangular matrix $D$ of $M$}
\label{Alg4}
\begin{algorithmic}[1]
\Require{Square matrix $M=(a_{ij})$ of order $n$, with entries in $\mathbb{F}$}
\Ensure{Upper triangular matrix $D$}
\State{$j\gets 0$}
\Repeat
\State{$j\gets j+1$}
\State{$i\gets 0$, $k\gets 0$}
\Repeat
\State{$i\gets i+1$}
\If{ $a_{ij}\ne 0$}
\State{$B_i=(b_{ij})\gets \begin{pmatrix} a_{1j}&a_{2j}&\cdots&a_{nj} \end{pmatrix}$}
%
%
%
%
\For{$i_1\gets 1$ to $n$ and $i_1\ne i$}
\State{$C_{i_1}\gets \left( \begin{array}{c c}  (a_{i_1 1}-a_{i_1 j}\cdot a_{ij}^{-1}\cdot b_{11} )& (a_{i_1 2}-a_{i_1 j}\cdot a_{ij}^{-1}\cdot b_{1 j}) \cdots \end{array}\right.$
\phantom{sasasasasasasasasasasasasasasasasasasasasasasasasa!}
$\left.\begin{array}{c c} \cdots & (a_{i_1 n}-a_{i_1 j}\cdot a_{ij}^{-1}\cdot b_{1n})\end{array} \right)$}
\EndFor
\State{$D\gets \left( \begin{array}{c} D \\ \hline B_i \end{array}\right)$}
\State{$M=(a_{ij})\gets \left( \begin{array}{c} C_1\\ \hline \vdots \\ \hline C_n\end{array}\right)$}
\State{Let $m$ be the number of rows of $M$.}
\Else 
\State{$k\gets k+1$}
\EndIf
\If{$k=m$ and $k\ne 0$}
\State{$B_i\gets \begin{pmatrix} 0&0&\cdots&0\end{pmatrix}$}
\State{$D\gets \left( \begin{array}{c} D \\ \hline B_i \end{array}\right)$}
\EndIf
\Until{$i\geq m$}
\Until{$j=n$}\\
\Return{$D$}
\end{algorithmic}
\end{algorithm}

\medskip

As an application of Algorithm \ref{Alg4}, we give here a Magma program to compute the upper triangular matrix of \eqref{M} with entries in the real quaternion division ring $\mathbb{H}$.

\medskip

Defining before the division ring $\mathbb{H}$,
\begin{verbatim}
F<i,j,k> := QuaternionAlgebra< RealField() | -1, -1 >;
\end{verbatim}
we have the following Magma program:
\begin{prog}\label{prog 7} ~
\begin{verbatim}
MT:=function(M)
n:=NumberOfRows(M); m:=NumberOfRows(M);
MM:=RemoveRow(SubmatrixRange(M,1,1,1,n),1); 
j:=0;
repeat
j:=j+1; i:=0; k:=0;
 repeat
 i:=i+1;
  if M[i,j] ne 0 then
   a:=M[i,j]; M1:=SubmatrixRange(M,i,1,i,n); M4:=M1; M2:=RemoveRow(M,i); 
   n1:=NumberOfRows(M2);
    for i1 in [1..n1] do
     M3:=Matrix(F,1,n,[ M2[i1,j1]-M2[i1,j]*(1/a)*M1[1,j1] : j1 in [1..n]]);
     M4:=VerticalJoin(M4,M3);
    end for;
   MM:=VerticalJoin(MM,M1); M:=RemoveRow(M4,1); m:=NumberOfRows(M);
   else 
   k:=k+1;
  end if;
 if k eq m and k ne 0 then
  MM:=VerticalJoin(MM,ZeroMatrix(F,1,n));
 end if;
 if k eq n then
  j:=n;
 end if;
until i ge m;
until j eq n;
return MM;
end function;
\end{verbatim}
\end{prog}
So, by typing in Magma
\begin{verbatim}
MT(Matrix(F,7,7,[-i,-j,0,k,1,0,0,0,-i,j,0,-k,1,0,0,0,-i,-j,0,k,1,j,0,0,
1,0,0,0,0,-j,0,0,1,0,0,0,0,j,0,0,1,0,0,0,0,-j,0,0,1]));
\end{verbatim}
\smallskip
we obtain the upper triangular matrix $E$ of \eqref{M},
\begin{equation}{\label{(7)}}
{\tiny
E=\begin{pmatrix}
-i & -j & 0 & k& 1& 0& 0 \\
0 & -i & j & 0& -k& 1& 0 \\
0& 0 & -i & -j& 0& k& 1\\
0 & 0 & 0 & i& 0& 0& k\\
0 & 0 & 0 & 0& 0& 0& 0\\
0 & 0 & 0 & 0& 0& 0& 0 \\
0 & 0 & 0 & 0& 0& 0& 0 \\
\end{pmatrix}.}
\end{equation}

Now, let us recall that the Dieudonn\'e determinant of an upper (or lower) triangular matrix $D$ with entries in a division ring $\mathbb{F}$ is the coset $a[\mathbb{F}^*,\mathbb{F}^*]$, where $a$ is the product of the elements of the main diagonal of $D$ (see \cite[p.~104]{draxl2}). Having in mind this, the above Algorithm \ref{Alg4} together with the next Algorithm \ref{Alg5} allow us to calculate up to a sign the Dieudonn\'e determinant of any square matrix with entries in $\mathbb{F}$.


\begin{algorithm}[h!] 
\small
\caption{Computation of Dieudonn\'e determinant of an upper triangular matrix.}
\label{Alg5}
\begin{algorithmic} [1]
\Require{Upper triangular matrix $M$.}
\Ensure{Dieudonn\'e determinant of $M$}
\State{Let $M=(a_{ij})$ be the upper triangular matrix}
\State{$A\gets 1$}
\For{$n\gets 1$ to $n$}
\State{$A\gets A\cdot a_{nn}$}
\EndFor
\If{$A=0$}\\
\Return{Dieudonn\'e determinant is $0$}
\Else
\If{$A\in [\mathbb{F}^*,\mathbb{F}^*]$}\\
\Return{Dieudonn\'e determinant is $0$}
%
%
%
\Else\\
\Return{Dieudonn\'e determinant is $A\;(\text{mod}\;[\mathbb{F}^*,\mathbb{F}^*])$}
\EndIf
\EndIf
\end{algorithmic}
\end{algorithm}

\smallskip

Finally, using the function ``MT" of Program \ref{prog 7} and having in mind that $[\mathbb{H}^{*},\mathbb{H}^{*}]=\{q\in \mathbb{H}:|q|=1\}$ (see \cite[Lemma 8, p.~151 ]{commutatorQ}), the following Magma test allows us to check if the Dieudonn\'e determinant of a matrix with entries in $\mathbb{H}$ is zero or not. 

\smallskip

\begin{prog} \label{Program:Dieudonne} ~
\begin{verbatim}
DD:=function(M)
MM:=MT(M);
A:=1;
n:=NumberOfRows(M);
 for N in [1..n] do
  A:=A*MM[N,N];
 end for;
 if A*Conjugate(A) eq F!1 or A*Conjugate(A) eq F!0 then
  print"Ddet is";
  return 0;
 end if;
print"Ddet is NOT";
return 0;
end function;
\end{verbatim}
\end{prog}

\bigskip

Let us give here a characterization of the degree of the $gcrd(f,g)$ which can be useful also to check condition 3) in Theorem \ref{Res1.5}.

\begin{theorem}\label{teor 3.7}
Let $\mathcal{P}_k(\mathbb{F})$ be the set of the polynomials in $\mathcal{R}$ of degree less than or equal to $k$ with coefficients in $\mathbb{F}$. Let $f,g\in\mathcal{R}$ be two polynomials of positive degree $m,n$ respectively. Consider the left $\mathbb{F}$-linear map 
$$\varphi : \mathcal{P}_{n-1}(\mathbb{F})\oplus\mathcal{P}_{m-1}(\mathbb{F})\to\mathcal{P}_{n+m-1}(\mathbb{F})$$
defined by $\varphi ((a,b)):=af+bg$. Then
$$\deg gcrd(f,g)=\dim \ker \varphi=\dim \ker \phi = n+m-lr.rk(A)=n+m-rc.rk(A)  \ ,$$
where $\phi: \mathbb{F}^{n+m}\to\mathbb{F}^{n+m}$ is the left $\mathbb{F}$-linear map given by $\phi(\vec{x}):=\vec{x}A$ with $A:=\text{Sylv}_{\mathbb{F}}^{\sigma,\delta}(f,g)$ the matrix defined in \eqref{For3.1} and $lr.rk(A)$ ($rc.rk(A)$) is the left row (right column) rank of $A$ which means the dimension of the $\mathbb{F}$-subspace spanned by the rows (columns) of $A$ viewed as elements of the $n+m$-dimensional left (right) vector space $\mathcal{P}_{n+m-1}(\mathbb{F})$ over $\mathbb{F}$.
\end{theorem}
\begin{proof}
The equality $\dim \ker \varphi=\dim \ker \phi$ can be obtained using the identification $\mathcal{P}_k(\mathbb{F})\cong\mathbb{F}^{k+1}$ given by the left $\mathbb{F}$-linear map $p_kx^k+\dots+p_1x+p_0\ \mapsto \ (p_k,\dots,p_1,p_0)$. Since $\mathcal{R}$ is a LPID, then we have
$$\mathcal{R}f+\mathcal{R}g=\mathcal{R}M \ , \ \mathcal{R}f\cap\mathcal{R}g=\mathcal{R}m\ ,$$
where $M:=gcrd(f,g)$ and $m:=lcrm(f,g)$ (least common right multiple). Then there are unique polynomials $\alpha,\beta\in\mathcal{R}$ such that $m=\alpha f=\beta g$. Moreover, by Lemma \ref{Res1.1} 3) we get also
$$\deg\alpha=\deg (m) - \deg f=(n+m-\deg M) - m=n-\deg M\ ,$$
$$\deg\beta=\deg (m) - \deg g=(n+m-\deg M) - n=m-\deg M\ .$$
Now, let $(a,b)\in\ker\varphi$. Hence $af=(-b)g\in\mathcal{R}m$. Thus there exists $t\in\mathcal{R}$ such that $af=(-b)g=tm$. This gives $af=t\alpha f$ and $(-b)g=t(-\beta)g$, i.e $a=t\alpha$ and $b=t\beta$. Therefore, by Lemma \ref{Res1.1} 3) we obtain that $(a,b)=(t\alpha,t\beta)$ with 
$$\deg t+(n-\deg M)=\deg (t\alpha)=\deg a\leq n-1\ ,$$ 
$$\deg t+(m-\deg M)=\deg (t\beta)=\deg b\leq m-1\ ,$$
that is, $\deg t\leq \deg M-1$ for both cases. This shows that
$$\ker \varphi\subseteq \left\{ (t\alpha,t\beta)\ :\ t\in\mathcal{R}\ ,\ \deg t\leq \deg M-1 \right\}\ .$$ Finally, let
$(t\alpha,t\beta)$ for some $t\in\mathcal{R}$ with $\deg t\leq \deg M-1$. Then $(t\alpha,t\beta)\in\mathcal{P}_{n-1}(\mathbb{F})\oplus\mathcal{P}_{m-1}(\mathbb{F})$ and $\varphi((t\alpha,t\beta))=t\alpha f+t\beta g=t(\alpha f+\beta g)=0$. Hence $(t\alpha,t\beta)\in\ker\varphi$ for some $t\in\mathcal{R}$ with $\deg t\leq \deg M-1$. This gives
$$\ker \varphi = \left\{ (t\alpha,t\beta)\ :\ t\in\mathcal{R}\ ,\ \deg t\leq \deg M-1 \right\}\ .$$
Observe that the set 
$$(\alpha,\beta),(x\alpha,x\beta),(x^2\alpha,x^2\beta),\dots,(x^{\deg M-1}\alpha,x^{\deg M-1}\beta)$$
is a left basis for $\ker \varphi$. Thus it follows that $\dim \ker \varphi = \deg M = \deg gcrd(f,g)$. Finally, since 
$\dim \mathrm{Im}(\phi) = lr.rk(A) = rc.rk(A)$,
by the rank-nullity theorem we obtain also that 
$\dim\ker\phi = n+m-\dim \mathrm{Im}(\phi)=n+m-lr.rk(A)=n+m-rc.rk(A)$.
\end{proof}

\medskip

\begin{remark}
Given a matrix $A$ over a division ring $\mathbb{F}$, it is known that the rank of $A$, denoted by  $rk(A):=lr.rk(A)=rc.rk(A)$, is equal to the number of all non-zero rows of the reduced-row echelon matrix of $A$ (see \cite[Theorem 1.3] {rank}). Thus, by Algorithm \ref{Alg4} we can easily compute $rk(\text{Sylv}_{\mathbb{F}}^{^{\sigma,\delta}}(f,g))$ (see Example \ref{ex 3.11}).
\end{remark}

\medskip

Here are some examples concerning Theorem \ref{Res1.5}.

\medskip

\begin{example} \label{Res1.7}
 Consider $\mathbb{F}_4[x;\sigma,\delta_t]$ with $\mathbb{F}_4=\{ 0,1,\alpha, \alpha^2\}$, where $\alpha^2+\alpha+1=0$, $\sigma(a)=a^2$ and $\delta_t(a)=t(\sigma(a)+a)$ for all $a \in \mathbb{F}_4$ and  $t\in\{0,1,\alpha,\alpha^2\}$.
 Given  $f_1:=x^2+\alpha^2x+\alpha$ and $g_1:=x^2+\alpha x+\alpha^2$, we have $$R^{\sigma,\delta_t}_{\mathbb{F}_4}(f_1,g_1)=\det \begin{pmatrix}
\alpha & \alpha^2 & 1 & 0 \\
t & \alpha^2+t & \alpha  & 1\\
\alpha^2 & \alpha & 1  & 0 \\
t & \alpha+t & \alpha^2 & 1 
\end{pmatrix}=0$$
This shows that $f_1$ and $g_1$ have a common (non-unit) right factor, independent of $t\in \mathbb{F}_4$. In fact, the common right factor is $(x+1)$, because $f_1=(x+\alpha)(x+1)$ and $g_1=(x+\alpha^2)(x+1)$.
On the other hand, consider $\delta_{\alpha}(a)=\alpha(\sigma(a)+a)$ and the skew polynomials \begin{center}
   $f_2:=(x+1)(x+\alpha)=x^2+\alpha x$ \;,\;\;$g_2:=(x+1)(x+\alpha^2)=x^2+\alpha^2 x+1.$ 
\end{center} Note that $(x+1)$ is a common (non-unit) left factor of $f_2$ and $g_2$, but this does not guarantee that $R^{\sigma,\delta_{\alpha}}_{\mathbb{F}_4}(f_2,g_2)$ is zero as in the commutative case. Indeed, we have $$R^{\sigma,\delta_{\alpha}}_{\mathbb{F}_4}(f_2,g_2)=\det \begin{pmatrix}
0 & \alpha & 1 & 0 \\
0 & \alpha & \alpha^2  & 1\\
1 & \alpha^2 & 1  & 0 \\
0 & \alpha^2 & \alpha & 1 
\end{pmatrix}=\alpha^2\neq 0.$$
\end{example}

\medskip

\begin{example}\label{Res1.8}
 Let $\mathbb{F}_{5}(t)$ be the field of rational functions over $\mathbb{F}_5$ and consider  $\mathbb{F}_{5}(t)[x;\sigma, \delta]$, where $\sigma:\mathbb{F}_{5}(t) \to \mathbb{F}_{5}(t), t \mapsto t^{5}$ ($\sigma$ is not an automorphism by Remark \ref{Pre2}) and $\delta$ is the classical derivation with respect to the variable $t$, i.e. $\delta:=\frac{d}{dt}$. Given $f_1:=\frac{1}{t}x(x+1)=\frac{1}{t}x^2+\frac{1}{t}{x}$ and $g_1:=(x+t^2)(x+1)=x^2+(t^2+1)x+t^2$, we have 
$$R^{\sigma,\delta}_{\mathbb{F}_5(t)}(f_1,g_1)=\det \begin{pmatrix}
0 & \frac{1}{t} & \frac{1}{t} & 0 \\
0 & \frac{4}{t^2} & \frac{1+4t^3}{t^5} & \frac{1}{t^5}\\
{t^2} & t^2+1 & 1  & 0 \\
2t & t^{10}+2t & t^{10}+1 & 1 
\end{pmatrix}=0.$$
Thus, $f_1$ and $g_1$ have a common right factor in $\mathbb{F}_{5}(t)[x;\sigma, \delta]$. On the other hand, if we consider $f_2:=(x+1)\frac{1}{t}x=\frac{1}{t^5}x^2+\left(\frac{t+4}{t^2}\right)x$ and $g_2:=(x+1)(x+t^2)=x^2+(t^{10}+1)x+(t^2+2t)$, having $(x+1)$ as a common left factor, we have $$R^{\sigma,\delta}_{\mathbb{F}_5(t)}(f_2,g_2)=\det \begin{pmatrix}
0 & \frac{t+4}{t^2} & \frac{1}{t^5} & 0 \\
0 & \frac{2+4t}{t^3} & \frac{t^5+4}{t^{10}} & \frac{1}{t^{25}}\\
{t^2}+2t & t^{10}+1 & 1  & 0 \\
2t+2 & t^{10}+2t^5 & t^{50}+1 & 1 
\end{pmatrix}=k\neq 0$$
where $k=\frac{1}{t^{30}}(4t^{56} + 4t^{55} + 2t^{54} + t^{26} + 2 t^{25} + 3 t^{24} +  t^{23} + 4t^{21} + 4t^{20} + 2 t^{19} + t^{12} + 3 t^{10} + 2 t^7 + 3 t^6 + t^5 + 2 t^4 + 3 t^3 + 3 t + 3)$.
\end{example}

\begin{example}
Consider $\mathbb{C}[x;\sigma,\delta]$, with $\sigma(z)=\bar{z}$ (the complex conjugation) and $\delta(z)=z-\bar{z}$, for all $z\in \mathbb{C}$. Given $f=x^4+(1+i)x^2-4ix+5i$ and $g=x^3-ix+2i$, we have $$R^{\sigma,\delta}_{\mathbb{C}}(f,g)=\det \begin{pmatrix}
5i & -4i & 1+i & 0 & 1& 0 & 0 \\
10i & -13i & 6i & 1-i & 0& 1 & 0\\
20i & -36i & 25i  & -8i & 1+i& 0 & 1\\
2i & -i & 0  & 1 & 0& 0 & 0\\
4i & -4i & i  & 0 & 1& 0 & 0\\
8i & -12i & 6i  & -i & 0& 1 & 0\\
16i & -32i & 24i  & -8i & i& 0 & 1\\
\end{pmatrix}=0.$$
Then, by Theorem \ref{Res1.5} we have  $\text{gcrd}(f,g)\neq 1$. By using the right division algorithm, we can find $\text{gcrd}(f,g)$ (as in the classical case): $$\begin{array}{rcl}
  x^4+(1+i)x^2-4ix+5i & = & x(x^3-ix+2i)+x^2+i
  \\x^3-ix+2i & = & x(x^2+i)+0
\end{array}$$
Hence $\text{gcrd}(f,g)=x^2+i$. 
\end{example}

\begin{example}{\label{ex 3.11}}
Consider $\mathbb{H}[x;\sigma,0]$, where $\sigma(h):=ihi^{-1}$ (inner automorphism) for all $h\in \mathbb{H}$.  Given $p=x^2+(i-j)x+k$ and $q=x+i$ in $\mathbb{H}[x;\sigma,0]$, we have
$$R^{\sigma,0}_{\mathbb{H}}(p,q)=\text{Ddet} \begin{pmatrix}
k & i-j & 1  \\
i & 1 & 0 \\
0 & i & 1  
\end{pmatrix}=0$$
Therefore $p$ and $q$ have a common (non-unit) right factor in $\mathbb{H}[x;\sigma,0]$, which must be $q=(x+i)$. In fact, $p=x^2+(i-j)x+k=(x-j)(x+i)$. Given now $f=x^4+kx^3-jx-i$ and $g=x^3+j$, 
the right $(\sigma,\delta)$-Sylvester matrix of $f(x)$ and $g(x)$ and its upper triangular matrix are \eqref{M} and \eqref{(7)}, respectively.
Hence we have $R^{\sigma,0}_{\mathbb{H}}(f,g)=0$.
Therefore $f$ and $g$ have a common (non-unit) right factor in $\mathbb{H}[x;\sigma,0]$, which must be $g=(x^3+j)$. In fact, $f=x^4+kx^3-jx-i=(x+k)(x^3+j)$. 
Moreover, note that the echelon form of ${\text{Sylv}}^{\sigma,0}_{\mathbb{H}}(f,g)$ is the matrix \eqref{(7)}. Therefore $rk(\text{{Sylv}}^{\sigma,0}_{\mathbb{H}}(f,g))=4$ and by Theorem \ref{teor 3.7} we have $\deg(gcrd(f,g))=3$. In fact, $gcrd(f,g)=x^3+j$.
\end{example}

\begin{remark}
In the commutative case, it is known that the last non-zero row of the Sylvester's matrix, when we put it in echelon
form by using only row transformations, gives the coefficients of the greatest common divisor (see \cite[Theorem 3]{laidacker}). However, this is not true for the noncommutative case. In fact, given $f,g\in \mathbb{H}[x;\sigma,0]$ as in Example \ref{ex 3.11}, the echelon form of ${\text{Sylv}}^{\sigma,0}_{\mathbb{H}}(f,g)$ is the matrix \eqref{(7)} and the entries of the last non-zero row of \eqref{(7)} are different from the coefficients of $\text{gcrd}(f,g)=x^3+j$.
\end{remark}

\medskip

Here are some basic properties of the right $(\sigma,\delta)$-resultant.

\begin{proposition}\label{basicpropert} 
Let $f,g\in\mathcal{R}$ be two skew polynomials of non-negative degrees $m$ and $n$, respectively. The following properties hold:
\begin{itemize}
    \item[$1)$] $R_{\mathbb{F}}^{\sigma,\delta}(g,f)=(-1)^{mn}R_{\mathbb{F}}^{\sigma,\delta}(f,g)$; 
    \item[$2)$] $R_{\mathbb{F}}^{\sigma,\delta}(-f,g)=(-1)^{n}R_{\mathbb{F}}^{\sigma,\delta}(f,g)$ and $R_{\mathbb{F}}^{\sigma,\delta}(f,-g)=(-1)^{m}R_{\mathbb{F}}^{\sigma,\delta}(f,g)$;
    \item[$3)$] if $g=x-a$, then $R_{\mathbb{F}}^{\sigma,\delta}(f,g)=0$ if and only if $f(a)=0$. In particular, for $a=0$ we have $R_{\mathbb{F}}^{\sigma,\delta}(f,g)=f(0)\;(\text{mod}\;[\mathbb{F}^*,\mathbb{F}^*])$;
    \item[$4)$] if $g=b_0$, then $R_{\mathbb{F}}^{\sigma,\delta}(f,g)=b_0\sigma(b_0)\sigma^2(b_0)\cdots \sigma^{m-1}(b_0)\;(\text{mod}\; [\mathbb{F}^*,\mathbb{F}^*])$;    
    \item[$5)$] if $\delta=0$ and $c\in \mathbb{F}^{*}$, then $R_{\mathbb{F}}^{\sigma,0}(cf,g)=N_n^{\sigma,0}(c)(\text{mod}\; [\mathbb{F}^*,\mathbb{F}^*])\ R_{\mathbb{F}}^{\sigma,0}(f,g)\;$. 
    \end{itemize}
\end{proposition}
\begin{proof}
$1)$ The $(\sigma,\delta)$-resultant $R^{\sigma,\delta}_{\mathbb{F}}(g,f)$ is obtained by permuting the rows of the Sylvester matrix $\text{Sylv}^{\sigma,\delta}_{\mathbb{F}}(f,g)$. The number of permutations is $mn$ and, since the exchange of two any rows of a matrix changes the sign of the Dieudonn\'e determinant, it follows that $R_{\mathbb{F}}^{\sigma,\delta}(g,f)=(-1)^{mn}R_{\mathbb{F}}^{\sigma,\delta}(f,g)$.

$2)$ By the properties of the Dieudonn\'e determinant, if a row of a matrix is left multiplied by $a\in \mathbb{F}^{*}$, then Ddet is left multiplied by $a\ (\text{mod}\; [\mathbb{F}^*,\mathbb{F}^*])$. Thus, since the first $n$ rows of $\text{Sylv}_{\mathbb{F}}^{\sigma,\delta}(f,g)$ contain the coefficients of $f$, $\sigma(-a)=-\sigma(a)$ and $\delta(-a)=-\delta(a)$, it follows that $R_{\mathbb{F}}^{\sigma,\delta}(-f,g)=(-1)^{n}R_{\mathbb{F}}^{\sigma,\delta}(f,g)$. Similarly, we get $R_{\mathbb{F}}^{\sigma,\delta}(f,-g)=(-1)^{m}R_{\mathbb{F}}^{\sigma,\delta}(f,g)$.

$3)$ It follows easily from the equivalence between $1)$ and $2)$ of Theorem \ref{Res1.5}.

$4)$ If $g=b_0$, then $\text{Sylv}^{\sigma,\delta}_{\mathbb{F}}(f,g)$ is a lower triangular matrix whose elements on the main diagonal are $b_0,\sigma(b_0),\sigma^2(b_0),..., \sigma^{m-1}(b_0)$. Then, the statement holds because $\text{Ddet}$ of a lower (or upper) triangular matrix is the coset $a[\mathbb{F}^*,\mathbb{F}^*]$, where $a$ is the (ordered) product of the elements on the main diagonal (see \cite[p.~104]{draxl2}).

$5)$ Since $\delta=0$ and $\sigma$ is an endomorphism of $\mathbb{F}$, we have

\begin{center}
\resizebox{15.8 cm}{2.3 cm}{
\begin{tabular}{c}
$\text{Sylv}^{\sigma,0}_{\mathbb{F}}(cf,g)=\begin{pmatrix}
ca_{0} & ca_{1} & ca_2 & \cdots  & ca_m & 0 & \cdots & 0\\
0 & \sigma(c)\sigma(a_{0}) & \sigma(c)\sigma(a_{1})  & \cdots & \sigma(c)\sigma(a_{m-1}) & \sigma(c)\sigma(a_m) & \cdots & 0\\
\vdots & \vdots & \ddots  & \vdots & \vdots & \vdots & \ddots & \vdots &\\
0 & 0 & 0 & \cdots & \sigma^{n-1}(c)\sigma^{n-1}(a_0) & \sigma^{n-1}(c)\sigma^{n-1}(a_{1}) & \cdots & \sigma^{n-1}(c)\sigma^{n-1}(a_m) \\
b_{0} & b_{1} & b_2 & \cdots  & b_n & 0 & \cdots & 0\\
0 & \sigma(b_{0}) & \sigma(b_{1})  & \cdots & \sigma(b_{n-1}) & \sigma(b_n) & \cdots & 0\\
\vdots & \vdots & \ddots  & \vdots & \vdots & \vdots & \ddots & \vdots &\\
0 & 0 & 0 & \cdots & \sigma^{m-1}(b_0) & \sigma^{m-1}(b_{1}) & \cdots & \sigma^{m-1}(b_n) \ .
\end{pmatrix}$
\end{tabular}
}
\end{center}
Noting that the first $n$ rows of the above matrix are multiplied on the left by
$c,\sigma(c),\sigma^2(c)$, $...,\sigma^{n-1}(c)$, respectively, it follows that $$R_{\mathbb{F}}^{\sigma,0}(cf,g)=\sigma^{n-1}(c)\cdots \sigma(c) c\;(\text{mod}\; [\mathbb{F}^*,\mathbb{F}^*]) R_{\mathbb{F}}^{\sigma,0}(f,g)= N^{\sigma,0}_n(c)\;(\text{mod}\; [\mathbb{F}^*,\mathbb{F}^*]) R_{\mathbb{F}}^{\sigma,0}(f,g)\ .$$ 
\end{proof}

\begin{remark}
The known property of ``factorization" of the classical resultants, that is, $R(f_1f_2,g)=R(f_1,g)\cdot R(f_2,g)$, is not true in general for our notion of resultant. Indeed, if we consider the ring $\mathbb{C}[x;\sigma,\delta]$, with $\sigma(z)=\bar{z}$ and $\delta(z)=z-\bar{z}$, for all $z\in \mathbb{C}$ and the skew polynomials $f_1=x^2+1$, $f_2=x^2+i$ and $g=2x^2+x+1$, we have 
\begin{center}
   $R_{\mathbb{C}}^{\sigma,\delta}(f_1,g)\cdot R_{\mathbb{C}}^{\sigma,\delta}(f_2,g) =\det \begin{pmatrix}
1 & 0 & 1 & 0   \\
0 & 1 & 0 & 1  \\
1 & 1 & 2  & 0  \\
0& 1 & 1 & 2  \\
\end{pmatrix}\cdot \det \begin{pmatrix}
i & 0 & 1 & 0   \\
2i & -i & 0 & 1  \\
1 & 1 & 2  & 0  \\
0& 1 & 1 & 2  \\
\end{pmatrix}=10+10i$
\end{center} 
However, 
\begin{center}
    $R_{\mathbb{C}}^{\sigma,\delta}(f_1f_2,g)=\det \begin{pmatrix}
5i & -4i & 1+i & 0 & 1& 0  \\
10i & -13i & 6i & 1-i & 0& 1 \\
1 & 1 & 2  & 0 & 0& 0 \\
0& 1 & 1 & 2 & 0& 0 \\
0 & 0 & 1  & 1 & 2& 0 \\
0 & 0 & 0  & 1 & 1& 2 \\
\end{pmatrix}= 650+90i.$
\end{center}
This shows that in general $R_{\mathbb{F}}^{\sigma,\delta}(f_1f_2,g)\neq R_{\mathbb{F}}^{\sigma,\delta}(f_1,g)\cdot R_{\mathbb{F}}^{\sigma,\delta}(f_2,g)$, also when $\delta=0$.
\end{remark}

\medskip

\begin{lemma}[Cramer's Rule]\label{Cramer}
Let $A$ be a non-singular square matrix $n\times n$ with entries in $\mathbb{F}$ and consider the linear system $A\cdot\overline{x}=\overline{b}$ for some column vector $\overline{b}$, where $\overline{x}$ is the transpose of the unknown vector $(x_1,x_2,\dots,x_n)$. If we write $A=[\overline{v_1}|\dots |\overline{v_n}]$, where the $\overline{v_i}$'s are the columns of $A$, then we have
for $i=1,\dots ,n$
$$x_i\;(\text{mod}\; [\mathbb{F}^*,\mathbb{F}^*]) =(\text{Ddet} (A))^{-1}\text{Ddet} (A_i)  ,$$
where $A_i:=[\overline{v_1}|\dots |\overline{b}|\dots|\overline{v_n}]$ is the matrix $A$ with the $i$-th column $\overline{v_i}$ replaced by $\overline{b}$.
\end{lemma}
\begin{proof}
Write $A^{-1}\overline{v_j}=\overline{e_j}$ for $j=1,\dots ,n$, where the $\overline{e_j}$'s are the canonical column vectors. Then we have
$$A^{-1}\cdot\left[\overline{v_1}|\cdots |\overline{b}| \cdots |\overline{v_n}\right]=\left[\overline{e_1}|\cdots | \overline{x} |\cdots |\overline{e_n}\right]\ .$$
Therefore, by \cite[Theorem 4.5]{Br} we deduce that
$$x_i \;(\text{mod}\; [\mathbb{F}^*,\mathbb{F}^*])= \text{Ddet} \left[\overline{e_1}|\cdots | \overline{x} |\cdots |\overline{e_n}\right] = \text{Ddet} \left(A^{-1}\right)\text{Ddet} \left[\overline{v_1}|\cdots |\overline{b}| \cdots |\overline{v_n}\right]\ ,$$ obtaining the formula of the statement having in mind that $\text{Ddet}\left(A^{-1}\right)=\left(\text{Ddet}A\right)^{-1}$.
\end{proof}

\begin{remark}\label{Cramer bis}
In a similar way as in Lemma \ref{Cramer}, one can obtain the following row version of the Cramer's Rule. 
Let $B$ be a square matrix $n\times n$ with entries in $\mathbb{F}$ and consider the linear system $\overline{y}\cdot B=\overline{c}$ for some row vector $\overline{c}$, where $\overline{y}$ is the unknown vector $(y_1,y_2,\dots,y_n)$. If we denote by $\overline{w_j}$ the $j$-th row of $B$, then by \cite[Theorems 3.9 and 4.5]{Br} we have
for $j=1,\dots ,n$
$$y_j\;(\text{mod}\; [\mathbb{F}^*,\mathbb{F}^*])=\text{Ddet} B_j(\text{Ddet} B)^{-1} \ ,$$
where $B_j$ is the matrix $B$ with the $j$-th row $\overline{w_j}$ replaced by $\overline{c}$.  
\end{remark}

\begin{proposition}\label{pro3.19}
Let $f,g\in\mathcal{R}$ be two skew polynomials of positive degree. Then, there are $A,B\in\mathcal{R}$ such that
$$Af+Bg=R_{\mathbb{F}}^{\sigma,\delta}(f,g) \ ,$$
where the coefficients of $A$ and $B$ $(\text{mod}\; [\mathbb{F}^*,\mathbb{F}^*])$ are integer polynomials in the entries of $\text{Sylv}_{\mathbb{F}}^{\sigma,\delta}(f,g)$.
\end{proposition}
\begin{proof}
Assume that $R_{\mathbb{F}}^{\sigma,\delta}(f,g)\neq 0$, otherwise we are done by choosing $A=B=0$. Let
$$f=a_0x^l+\dots+a_l\ , \ \ a_0\neq 0 \ ,$$
$$g=b_0x^m+\dots+b_m\ , \ \ b_0\neq 0 \ ,$$
$$A'=c_0x^{m-1}+\dots+c_{m-1}\ , $$
$$B'=d_0x^{l-1}+\dots+d_{l-1}\ , $$
such that $A'f+B'g=1$, where the coefficients $c_0,\dots,c_{m-1},d_0,\dots,d_{l-1}$ are unknowns in $\mathbb{F}$. If we compare coefficients of powers of $x$ in the formula $A'f+B'g=1$, then we get the following system of linear equations similar to \eqref{For3.1} with unknowns $c_i,d_i$: 
\begin{equation}\label{syst}
(c_{m-1},\dots,c_0,d_{l-1},\dots,d_0)\cdot \text{Sylv}_{\mathbb{F}}^{\sigma,\delta}(f,g)=(0,\dots,0,1) \ .
\end{equation}
By applying Remark \ref{Cramer bis} to the square linear system \eqref{syst}, we obtain that all the $c_i$'s and the $d_i$'s are as follow:
$$c_i\;(\text{mod}\; [\mathbb{F}^*,\mathbb{F}^*])=R_{\mathbb{F}}^{\sigma,\delta}(f,g)^{-1}\text{Ddet}\;\text{Sylv}_{\mathbb{F}}^{\sigma,\delta}(f,g)_{m-i} \ , \ \textrm{for} \ i=0,\cdots,m-1 \ , $$ 
$$d_j\;(\text{mod}\; [\mathbb{F}^*,\mathbb{F}^*])=R_{\mathbb{F}}^{\sigma,\delta}(f,g)^{-1}\text{Ddet}\;\text{Sylv}_{\mathbb{F}}^{\sigma,\delta}(f,g)_{m+l-j} \ , \ \textrm{for} \ j=0,\cdots,l-1 \ , $$
where $\text{Sylv}_{\mathbb{F}}^{\sigma,\delta}(f,g)_k$ is the matrix $\text{Sylv}_{\mathbb{F}}^{\sigma,\delta}(f,g)$ with the $k$-th row replaced by the row vector $(0,\dots,0,1)$. Defining $A:=R_{\mathbb{F}}^{\sigma,\delta}(f,g)A', B:=R_{\mathbb{F}}^{\sigma,\delta}(f,g)B'$, we see that 
$Af+Bg=R_{\mathbb{F}}^{\sigma,\delta}(f,g)$ and the coefficients of $A$ and $B$ $(\text{mod}\; [\mathbb{F}^*,\mathbb{F}^*])$ are given by expressions of type $\text{Ddet}\;\text{Sylv}_{\mathbb{F}}^{\sigma,\delta}(f,g)_{h}$ for some $h=1,\dots,m+l$. We conclude by noting that these latest expressions are simply integer polynomials in the entries of $\text{Sylv}_{\mathbb{F}}^{\sigma,\delta}(f,g)$. 
\end{proof}

\medskip

Now, let us show that under certain conditions, it is possible to add a sixth equivalent condition in Theorem \ref{Res1.5}. To do that, we first need to introduce the following definition.

\begin{definition}\label{Res1.11}
We say that $\Tilde{\mathbb{F}}[x;\Tilde{\sigma},\Tilde{\delta}]$ is a \textit{polynomial ring extension} of $\mathcal{R}$ if $\mathbb{F}$ is a subring of $ \Tilde{\mathbb{F}}$, $\Tilde{\sigma}_{|{\mathbb{F}}}=\sigma$ and $\Tilde{\delta}_{|{\mathbb{F}}}=\delta$.
\end{definition}
  
\begin{remark}
Since $\mathcal{R}\subseteq \Tilde{\mathbb{F}}[x;\Tilde{\sigma},\Tilde{\delta}]$, $\Tilde{\sigma}_{|{\mathbb{F}}}=\sigma$ and $\Tilde{\delta}_{|{\mathbb{F}}}=\delta$, it is evident that $\mathcal{R}$ is closed with respect to the sum and the product of polynomials in $\Tilde{\mathbb{F}}[x;\Tilde{\sigma},\Tilde{\delta}]$. Moreover, since $\mathcal{R}$ contains the multiplicative identity of $\Tilde{\mathbb{F}}[x;\Tilde{\sigma},\Tilde{\delta}]$ (because $\mathbb{F}$ is a subring of $\tilde{\mathbb{F}}$), it follows that $\mathcal{R}$ is a subring of $ \Tilde{\mathbb{F}}[x;\Tilde{\sigma},\Tilde{\delta}]$. 
\end{remark}

Definition \ref{Res1.11} is motivated by the following situation. 
     
\medskip 

\noindent Consider $\mathbb{C}[x,\sigma,\delta]$, where $\sigma$ is the complex conjugation and $\delta$ is an inner derivation given by $\delta(z)=z-\sigma(z)=2\text{Im}(z)i$, for all $z \in \mathbb{C}$. Note that the skew polynomial $f=x^2+i$ has no right roots in $\mathbb{C}$. In fact, for all $z\in \mathbb{C}$, we have \begin{center}
    $f(z)=N_2(z)+iN_0(z)=|z|^2+(2\text{Im}(z)+1)i\neq 0.$
\end{center} 
The natural question is then the following: where does $f$ have a right root? Unlike the classical case, i.e. when $\sigma=\text{Id}$ and $\delta=0$, it will not be sufficient to extend $\mathbb{C}$ to find a right root of $f$, but it will be necessary to extend also the maps $\sigma$ and $\delta$, because the evaluation of $f$ at such a root will depend on the action of these new functions. Therefore, we need to construct a polynomial ring extension of $\mathbb{C}[x;\sigma,\delta]$ for finding a right root of $f$.
More in general, we will construct a polynomial ring extension $\Tilde{\mathbb{F}}[x;\Tilde{\sigma},\Tilde{\delta}]$ of $\mathbb{C}[x;\sigma,\delta]$ such that any irreducible skew polynomial $g\in \mathbb{C}[x;\sigma,\delta]$  has a right root in $\Tilde{\mathbb{F}}[x;\Tilde{\sigma},\Tilde{\delta}]$. 
Let $\mathbb{H}$ be the division ring of real quaternions. If we define over $\mathbb{H}$ the maps 
$\tilde{\sigma}(t):=a-bi+cj-dk$ and  $\tilde{\delta}(t):=t-\tilde{\sigma}(t)$, for all $t:=a+bi+cj+dk\in \mathbb{H}$,   it follows that $\tilde{\sigma}$ is an automorphism of $\mathbb{H}$ such that $\tilde{\sigma}_{|\mathbb{C}}=\sigma$ and $\tilde{\delta}$ is a $\tilde{\sigma}$-derivation such that $\tilde{\delta}_{|\mathbb{C}}=\delta$. Thus, $\mathbb{H}[x;\tilde{\sigma},\tilde{\delta}]$ is a polynomial ring extension of $\mathbb{C}[x;\sigma,\delta]$. Moreover, since every non-constant skew polynomial $g\in \mathbb{H}[x;{\sigma},{\delta}]$ splits into linear factors in $\mathbb{H}[x;\sigma,\delta]$ independently of $\sigma$ and $\delta$ (see \cite[Corollary 3]{closureH}), it follows that $g$ has all its roots in $\mathbb{H}[x;\sigma,\delta]$. In particular, the skew polynomial $f=x^2+i\in \mathbb{C}[x;\sigma,\delta]$ will have its roots in $\mathbb{H}[x;\tilde{\sigma},\tilde{\delta}]$. 
Therefore, $\mathbb{H}[x;\tilde{\sigma},\tilde{\delta}]$ it looks like as a ``\textit{closure}" of $\mathbb{C}[x;\sigma,\delta]$.

\medskip
      
\begin{remark}\label{Res1.14}
Let $\mathbb{F}=\mathbb{F}_q$ be a finite field with $q$ elements, where $q=p^m$ for some prime $p$ and $m\in \mathbb{Z}_{\geq 1}$. Given any field extension $\tilde{\mathbb{F}}/\mathbb{F}_q$ and any automorphism $\sigma$ of $\mathbb{F}_q$, that is, $\sigma(a):=a^{p^j}$ for any $a\in \mathbb{F}_q$ and some integer $j$ such that $1\leq j \leq m$, we see that one can always extend trivially $\sigma$ to an automorphism  $\Tilde{\sigma}:\tilde{\mathbb{F}}\to \tilde{\mathbb{F}}$ such that $\Tilde{\sigma}|_{\mathbb{F}_q}=\sigma$ by defining $\tilde{\sigma}(b):=b^{p^{j}}$ for any $b\in \tilde{\mathbb{F}}_{q}$. Moreover, 
since any $\sigma$-derivation $\delta$ is an inner derivation (see Proposition \ref{inner derivations}), that is, $\delta_{\beta}(a):=\beta(\sigma(a)-a)$ for any $a\in \mathbb{F}_q$ and some $\beta \in \mathbb{F}_q$, we can also extend trivially $\delta_{\beta}$ to a $\tilde{\sigma}$-derivation $\tilde{\delta}_{\beta}:\tilde{\mathbb{F}}\to \tilde{\mathbb{F}}$ such that $\tilde{\delta_{\beta}}_{|\mathbb{F}_q}=\delta_{\beta}$ by defining $\Tilde{\delta}_{\beta}(b):=\beta(\Tilde{\sigma}(b)-b)$ for any $b\in \tilde{\mathbb{F}}$. This gives a special polynomial ring extension $\tilde{\mathbb{F}}[x;\Tilde{\sigma},\Tilde{\delta}]$ of $\mathbb{F}_q[x;\sigma,\delta]$ such that $N_{i}^{\tilde{\sigma},\tilde{\delta}}(y)=N_{i}^{\sigma,\delta}(y)$ for any $y\in \tilde{\mathbb{F}}$ and $i\in \mathbb{Z}_{\geq 0}$.
\end{remark}
  
\medskip    

\noindent The above remark shows that if $\mathbb{F}$ is a finite division ring (i.e. a finite field), then we can always construct a suitable polynomial ring extension. So, by Remark \ref{Res1.14} we can obtain the following result.
 
\begin{theorem}\label{Res1.15} 
Two non-constant skew polynomials $f,g\in \mathbb{F}_q[x;\sigma,\delta]$ have a common right root in some polynomial ring extension $\tilde{\mathbb{F}}[x;\Tilde{\sigma},\Tilde{\delta}]$ of $\mathbb{F}_q[x;\sigma,\delta]$ if and only if $\text{R}^{\sigma,\delta}_{\mathbb{F}_q}(f,g)=0$.
\end{theorem}

\begin{proof}
If $f(x)$ and $g(x)$ have a common right root $\alpha$ in some polynomial ring extension $\tilde{\mathbb{F}}[x;\Tilde{\sigma},\Tilde{\delta}]$ of $\mathbb{F}_q[x;\sigma,\delta]$, then  $f(x)=f_1(x)(x-\alpha)$ and $g(x)=g_1(x)(x-\alpha)$, for some $f_1(x),g_1(x) \in \tilde{\mathbb{F}}[x;\tilde{\sigma},\tilde{\delta}]$. By Theorem \ref{Res1.5}, since $f(x)$ and $g(x)$ have a common (non-unit) right factor $(x-\alpha)$ in $\tilde{\mathbb{F}}[x;\tilde{\sigma},\tilde{\delta}]$ it follows that $\text{R}^{\tilde{\sigma},\tilde{\delta}}_{\tilde{\mathbb{F}}}(f,g)=0$. However, since $\Tilde{\sigma}_{|{\mathbb{F}_{q}}}=\sigma$ and $\Tilde{\delta}_{|{\mathbb{F}_{q}}}=\delta$, we obtain that $\text{R}^{\sigma,\delta}_{\mathbb{F}_{q}}(f,g)=\text{R}^{\tilde{\sigma},\tilde{\delta}}_{\tilde{\mathbb{F}}}(f,g)=0$.
Conversely, if $\text{R}^{\sigma,\delta}_{\mathbb{F}_q}(f,g)=0$ then $f$ and $g$ have a common (non-unit) right factor $h(x):=\sum_{i}h_ix^{i}\in \mathbb{F}_{q}[x;\sigma,\delta]$. Thus, we can write $f(x)=f'(x)h(x)$ and $g(x)=g'(x)h(x)$, for some $f'(x),g'(x)\in \mathbb{F}_{q}[x;\sigma,\delta]$. If $h(x)$ has a right root in $\mathbb{F}_{q}$ then we are done. Otherwise, since any endomorphism $\sigma$ of $\mathbb{F}_{q}$ is an automorphism of the form $\sigma(a)=a^{p^{j}}$ for some integer $j$ such that $1\leq j \leq m$ and each $\delta$ is an inner derivation, observe that $\sum_{i}h_iN_{i}^{\sigma,\delta}(y)\in \mathbb{F}_{q}[y]$. Therefore, from classical field theory it follows that there exists a field extension $\tilde{\mathbb{F}}$ of $\mathbb{F}_{q}$ such that $\sum_{i}h_iN_{i}^{\sigma,\delta}(\tilde{y})=0$ for some $\tilde{y}\in \tilde{\mathbb{F}}$. So considering the special polynomial ring extension $\tilde{\mathbb{F}}[x;\Tilde{\sigma},\Tilde{\delta}]$ of $\mathbb{F}_{q}[x;\sigma,\delta]$ of Remark \ref{Res1.14} with $\tilde{\mathbb{F}}$ as above, we have $h(\tilde{y})=\sum_{i}h_iN_{i}^{\tilde{\sigma},\tilde{\delta}}(\tilde{y})=\sum_{i}h_iN_{i}^{\sigma,\delta}(\tilde{y})=0$ in $\tilde{\mathbb{F}}[x;\Tilde{\sigma},\Tilde{\delta}]$. Since $\mathbb{F}_{q}[x;\sigma,\delta]\subseteq \tilde{\mathbb{F}}[x;\tilde{\sigma},\tilde{\delta}]$, we can conclude by Theorem \ref{F1} that there exists $\tilde{y}\in \tilde{\mathbb{F}}$ such that $f(\tilde{y})=g(\tilde{y})=0$ in $\tilde{\mathbb{F}}[x;\tilde{\sigma},\tilde{\delta}]$, that is, $f$ and $g$ have a common right root in some polynomial ring extension of $\mathcal{R}$.
\end{proof}

\begin{corollary}\label{Res1.16}
 Let $\mathbb{F}$ be a division ring and let $f,g\in\mathcal{R}$ be two non-constant skew polynomials. If $f$ and $g$ have a common right root in some polynomial ring extension $\tilde{\mathbb{F}}[x;\Tilde{\sigma},\Tilde{\delta}]$ of $\mathcal{R}$, then $\text{R}^{\sigma,\delta}_{\mathbb{F}}(f,g)=0$. 
\end{corollary}
 
\begin{proof}
 It follows easily from the first part of the proof of Theorem \ref{Res1.15}.
\end{proof}
 
An interesting problem would be to determine, in general, when the reciprocal of Corollary \ref{Res1.16}  is true. We know that if $\text{R}^{\sigma,\delta}_{\mathbb{F}}(f,g)=0$ then $f$ and $g$ have a common (non-unit) right factor $h\in\mathcal{R}$. Then, the existence of a common right root between $f$ and $g$ is reduced to guarantee the existence of some polynomial ring extension where $h$ has a right root. If $\mathbb{F}=\mathbb{F}_q$ is a finite field, then we have seen in Theorem \ref{Res1.15} that for any skew polynomial $h\in \mathbb{F}_q[x;\sigma,\delta]$ we can find a polynomial ring extension where $h$ has a right root. If $\mathbb{F}=\mathbb{C}$, then for the case $\mathbb{C}[x;\sigma,\delta]$ with $\sigma$ the complex conjugation and $\delta$ an inner derivation, we know that $\mathbb{H}[x;\sigma,\delta]$ is a ``closure" of $\mathbb{C}[x;\sigma,\delta]$. Therefore, every irreducible skew polynomial $h\in \mathbb{C}[x;\sigma,\delta]$ has a right root in $\mathbb{H}[x;\sigma,\delta]$ and then the reciprocal of Corollary \ref{Res1.16} is true also in this case.

\medskip

Thus, one could ask in which other cases the reciprocal of Corollary \ref{Res1.16} is true. The following result gives a partial answer when $\mathbb{F}$ is an infinite division ring. 

\begin{proposition}\label{Res1.18}
Let $\mathbb{F}$ be an infinite division ring and  let $\sigma$ be an inner automorphism of $\mathbb{F}$. Skew polynomials $f,g\in \mathbb{F}[x;\sigma]=\mathbb{F}[x;\sigma,0]$ have a common right root in some polynomial ring extension $\tilde{\mathbb{F}}[x;\Tilde{\sigma}]$ of $\mathbb{F}[x;\sigma]$ if and only if $\text{R}^{\sigma,0}_{\mathbb{F}}(f,g)=0$.
\end{proposition}

\begin{proof}
By Corollary \ref{Res1.16}, the left-to-right implication is true. Conversely, suppose that 
$\text{R}^{\sigma,0}_{\mathbb{F}}(f,g)=0$. Then $f(x)=a(x)h(x)$ and $g(x)=b(x)h(x)$, for some $a(x),b(x),h(x)\in \mathbb{F}[x;\sigma]$ with $h(x):=\sum_{i=0}^{n}h_ix^{i}\in \mathbb{F}[x;\sigma]$ of positive degree. If $h(x)$ has a right root in $\mathbb{F}$, then we are done. Otherwise, since $\sigma$ is an inner automorphism, that is, $\sigma(a):=g^{-1}ag$ for all $a\in \mathbb{F}$ and $g\in \mathbb{F}^{*}$, we have 
\begin{center}
    $N_i^{\sigma,0}(a):=N_{i}(a)=g^{1-i}(ag)^{i}g^{-1}$, for all $a\in \mathbb{F}$, $i\in \mathbb{Z}_{\geq 0}.$
\end{center} 
Then, we get
\begin{center}
    $\sum_{i=0}^{n}h_iN_i(a)=\sum_{i=0}^{n}h_ig^{1-i}(ag)^ig^{-1}=(\sum_{i=0}^{n}h_ig^{1-i}(ag)^i)g^{-1}=(\sum_{i=0}^{n}h'_ib^i)g^{-1}$
\end{center}
where $b:=ag$ and $h'_i:=h_ig^{1-i}$ for all $i=0,1,...,n$. Since  $\sum_{i=0}^{n}h_iN_i(a)=0$ if and only if $\sum_{i=0}^{n}h'_ib^i=0$, it is sufficient to guarantee the existence of a right root of  $p(y):=\sum_{i=0}^{n}h'_{i}b^{i}\in \mathbb{F}[y]$. By \cite[Theorem 8.5.1]{cohnskew}, there exists a division ring extension (or skew field extension) $\tilde{\mathbb{F}}$ of $\mathbb{F}$ such that $p$ has a right root, say $p(\alpha)=\sum_{i=0}^{n}h'_i\alpha^{i}=0$ for some $\alpha \in \tilde{\mathbb{F}}$. Thus, defining $\tilde{\sigma}(z):=g^{-1}zg$ for all $z\in \tilde{\mathbb{F}}$ and putting $\beta:=\alpha g^{-1}\in \tilde{\mathbb{F}}$, we have $$h(\beta)=\sum_{i=0}^{n}h_iN_{i}^{\tilde{\sigma},0}(\beta)=\sum_{i=0}^{n}h_iN_{i}(\beta)=\left(\sum_{i=0}^{n}h'_{i}(\beta g)^{i} \right)g^{-1}=p(\beta g)g^{-1}=p(\alpha)g^{-1}=0$$
in $\tilde{\mathbb{F}}[x;\tilde{\sigma}]$. Since $\mathbb{F}[x;\sigma]\subseteq \tilde{\mathbb{F}}[x;\tilde{\sigma}]$, we can conclude again by Theorem \ref{F1} that there exists $\beta \in \tilde{\mathbb{F}}$ such that $f(\beta)=g(\beta)=0$ in $\tilde{\mathbb{F}}[x;\tilde{\sigma}]$, that is,  $f$ and $g$ have a common right root in some polynomial ring extension of $\mathbb{F}[x;\sigma]$. 
\end{proof}

\begin{remark}
Let $\mathbb{F}$ be an infinite division ring. If $\mathbb{F}$ is finite dimensional over its center $Z$, then every automorphism of $\mathbb{F}$ over $Z$ is inner (see \cite[Corollary 3.3.6]{cohnskew}). Therefore, under this hypothesis, the result of Proposition \ref{Res1.18} still holds. Moreover, if $\sigma, \delta$ are both inner, i.e. $\sigma (a):=g^{-1}ag$ and $\delta (a):=\sigma(a)v-va$ with $g\in\mathbb{F}^*$ and $v\in\mathbb{F}$ for all $a\in\mathbb{F}$, then by the change of variable $x':=x+v$, we have a ring isomorphism between $\mathcal{R}$ and $\mathbb{F}[x';\sigma]$ (see \cite[p. 295]{cohnfree}) which allows us to obtain a similar result as in Proposition \ref{Res1.18} for $\mathcal{R}$ and $\tilde{\mathbb{F}}[x;\tilde{\sigma},\tilde{\delta}]$, having in mind Corollary \ref{Res1.16}, Theorem \ref{Res1.5}, Proposition \ref{Res1.18} and the possibility of extending trivially $\sigma$ and $\delta$ as in its proof, that is, $\tilde{\sigma} (b):=g^{-1}bg$ and $\tilde{\delta} (b):=\tilde{\sigma} (b)v-vb$ for all $b\in\tilde{\mathbb{F}}$.
\end{remark}

\subsection{Left $(\sigma,\delta)$-Resultant}\label{LRes}
In Examples \ref{Res1.7} and \ref{Res1.8} we have seen that in general the condition $\text{R}^{\sigma,\delta}_{\mathbb{F}}(f,g)= 0$ is not related with the existence of common (non-unit) left factor of $f$ and $g$. From this, it seems interesting to study the possibility of defining a left $(\sigma,\delta)$-resultant which allows us to guarantee the existence of a common (non-unit) left factor for two skew polynomials. Since not every endomorphism $\sigma$ over a division ring $\mathbb{F}$ is an automorphism (see Remark \ref{Pre2}), we would like to emphasize the fact that in general the left-hand division of two skew polynomials cannot be performed in $\mathcal{R}$ (see e.g. \cite{Ore}). On the other hand, under the assumption that $\sigma$ is an automorphism, one can give a left-hand version of some of the main results shown in $\S$ \ref{RRes}.

\medskip

Keeping in mind that if $\sigma$ is an automorphism then $\mathcal{R}$ is a left Euclidean domain and hence a RPID (see \cite[Theorem 6]{Ore}), it is possible to give a left version of Lemma \ref{Res1.2} as follows.

\begin{lemma} \label{Res3.2}
Let $\sigma$ be an automorphism of $\mathbb{F}$. Two non-constant skew polynomials $f,g\in\mathcal{R}$ of respective degrees $m$ and $n$, have a common (non-unit) left factor in $\mathcal{R}$ if and only if there exist skew polynomials $c,d \in\mathcal{R}$ such that $fc+gd=0$, $\deg(c) < n$ and $\deg(d)< m$.
\end{lemma}
 
\noindent By Lemmas \ref{Res3.1} and  \ref{Res3.2}, we can define a left $(\sigma,\delta)$-resultant as follows.
Let \begin{align*}
f &= a_mx^m+...+a_1x+a_0, a_m \neq 0 \;,\;\;\;\; g=b_nx^n+...+b_1x+b_0, b_n \neq 0\;,\\
 c &= c_{n-1}x^{n-1}+...+c_1x+c_0 \;,\;\;\;\;\;\;\;\;\;\;\;\;d= d_{m-1}x^{m-1}+...+d_1x+d_0
\end{align*}
be skew  polynomials as in Lemma \ref{Res3.2}. By Lemma \ref{Res3.1}, we can write
\begin{align*}
f &= x^m A_m+...+xA_1+A_0, A_m \neq 0 \;,\;\;\;\; g=x^nB_n+...+xB_1+B_0, B_n \neq 0\;,\\
 c &= x^{n-1}C_{n-1}+...+xC_1+C_0 \;,\;\;\;\;\;\;\;\;\;\;\;\;d= x^{m-1}D_{m-1}+...+xD_{1}+D_0
\end{align*}
where $A_i, B_i,C_i,D_i$ are given by \eqref{For3.3}. Then, by \eqref{For1.3bis} we have
\begin{center}
$fc= \displaystyle{\sum_{i=0}^{m}\sum_{j=0}^{n-1}\left(\sum_{k=0}^{j}x^{i+j-k}(-1)^k \mathcal{T}_{k,j-k}(A_i)\cdot C_j\right)}$\\ $gd=\displaystyle{\sum_{i=0}^{n}\sum_{j=0}^{m-1}\left(\sum_{k=0}^{j} x^{i+j-k}(-1)^k \mathcal{T}_{k,j-k}(B_i)\cdot D_j\right)}$  
\end{center}
Thus the equation $fc+ gd=0$ of Lemma \ref{Res3.2} gives a homogeneous system of $m + n$ linear equations with $m + n$ unknowns $C_0,..., C_{n-1}, D_0,..., D_{m-1}$, that is 
\begin{equation} \label{For9}
M\cdot (C_0, ..., C_{n-1}, D_0, ..., D_{m-1})^{T}=(0,...,0)\;,
\end{equation}
where $M$ is the following $(m+n)\times(m+n)$ matrix:

\begin{center}
\resizebox{16,5 cm}{4cm}{
\begin{tabular}{c} 
$M=\begin{pmatrix}
A_0& A_1 & A_2 & \cdots & A_m & 0 & 0 & 0& \cdots  & 0 \\ 
-\mathcal{T}_{1,0}(A_0)& \displaystyle\sum_{i=0}^{1}(-1)^{1+i}\mathcal{T}_{1-i,i}(A_{1-i}) &  \displaystyle\sum_{i=0}^{1}(-1)^{1+i}\mathcal{T}_{1-i,i}(A_{2-i}) & \cdots & \displaystyle\sum_{i=0}^{1}(-1)^{1+i}\mathcal{T}_{1-i,i}(A_{m-i}) & \mathcal{T}_{0,1}(A_m)  & 0& 0 &\cdots & 0 \\ 
\mathcal{T}_{2,0}(A_0) &  \displaystyle\sum_{i=0}^{1}(-1)^{2+i}\mathcal{T}_{2-i,i}(A_{1-i}) & \displaystyle\sum_{i=0}^{2}(-1)^{2+i}\mathcal{T}_{2-i,i}(A_{2-i}) & \cdots & \displaystyle\sum_{i=0}^{2}(-1)^{2+i}\mathcal{T}_{2-i,i}(A_{m-i})& \displaystyle\sum_{i=1}^{2}(-1)^{2+i}\mathcal{T}_{2-i,i}(A_{m+1-i}) & \mathcal{T}_{0,2}(A_m) & 0 & \cdots &0 \\ 
-\mathcal{T}_{3,0}(A_0) &  \displaystyle\sum_{i=0}^{1}(-1)^{3+i}\mathcal{T}_{3-i,i}(A_{1-i}) & \displaystyle\sum_{i=0}^{2}(-1)^{3+i}\mathcal{T}_{3-i,i}(A_{2-i}) & \cdots & \displaystyle\sum_{i=0}^{3}(-1)^{3+i}\mathcal{T}_{3-i,i}(A_{m-i})& \displaystyle\sum_{i=1}^{3}(-1)^{3+i}\mathcal{T}_{3-i,i}(A_{m+1-i}) & \displaystyle\sum_{i=2}^{3}(-1)^{3+i}\mathcal{T}_{3-i,i}(A_{m+1-i}) & 
\mathcal{T}_{0,3}(A_m) & \cdots &0 \\
\vdots & \vdots & \vdots & \cdots  & \vdots & \vdots  & \vdots & \ddots  & \ddots& \vdots\\ 
(-1)^{n-1}\mathcal{T}_{n-1,0}(A_0) & \displaystyle\sum_{i=0}^{1}(-1)^{n-1+i}\mathcal{T}_{n-1-i,i}(A_{1-i}) & \displaystyle\sum_{i=0}^{2}(-1)^{n-1+i}\mathcal{T}_{n-1-i,i}(A_{2-i}) & \cdots & \displaystyle\sum_{i=0}^{m}(-1)^{n-1+i}\mathcal{T}_{n-1-i,i}(A_{m-i})& \displaystyle\sum_{i=1}^{n-1}(-1)^{n-1+i}\mathcal{T}_{n-1-i,i}(A_{m+1-i}) & \displaystyle\sum_{i=2}^{n-1}(-1)^{n-1+i}\mathcal{T}_{n-1-i,i}(A_{m+2-i}) & \displaystyle\sum_{i=3}^{n-1}(-1)^{n-1+i}\mathcal{T}_{n-1-i,i}(A_{m+3-i}) & \cdots & \mathcal{T}_{0,n-1}(A_m) \\
B_0& B_1 & B_2 & \cdots & B_n & 0 & 0 & 0& \cdots  & 0 \\ 
-\mathcal{T}_{1,0}(B_0)& \displaystyle\sum_{i=0}^{1}(-1)^{1+i}\mathcal{T}_{1-i,i}(B_{1-i}) &  \displaystyle\sum_{i=0}^{1}(-1)^{1+i}\mathcal{T}_{1-i,i}(B_{2-i}) & \cdots & \displaystyle\sum_{i=0}^{1}(-1)^{1+i}\mathcal{T}_{1-i,i}(B_{n-i}) & \mathcal{T}_{0,1}(B_n)  & 0& 0 &\cdots & 0 \\ 
\mathcal{T}_{2,0}(B_0) &  \displaystyle\sum_{i=0}^{1}(-1)^{2+i}\mathcal{T}_{2-i,i}(B_{1-i}) & \displaystyle\sum_{i=0}^{2}(-1)^{2+i}\mathcal{T}_{2-i,i}(B_{2-i}) & \cdots & \displaystyle\sum_{i=0}^{2}(-1)^{2+i}\mathcal{T}_{2-i,i}(B_{n-i})& \displaystyle\sum_{i=1}^{2}(-1)^{2+i}\mathcal{T}_{2-i,i}(B_{n+1-i}) & \mathcal{T}_{0,2}(B_n) & 0 & \cdots &0 \\ 
-\mathcal{T}_{3,0}(B_0) &  
\displaystyle\sum_{i=0}^{1}(-1)^{3+i}\mathcal{T}_{3-i,i}(B_{1-i}) & \displaystyle\sum_{i=0}^{2}(-1)^{3+i}\mathcal{T}_{3-i,i}(B_{2-i}) & \cdots & \displaystyle\sum_{i=0}^{3}(-1)^{3+i}\mathcal{T}_{3-i,i}(B_{n-i})& \displaystyle\sum_{i=1}^{3}(-1)^{3+i}\mathcal{T}_{3-i,i}(B_{n+1-i}) & \displaystyle\sum_{i=2}^{3}(-1)^{3+i}\mathcal{T}_{3-i,i}(B_{n+1-i}) & 
\mathcal{T}_{0,3}(A_m) & \cdots &0 \\
\vdots & \vdots & \vdots & \cdots  & \vdots & \vdots  & \vdots & \ddots  & \ddots& \vdots\\ 
(-1)^{m-1}\mathcal{T}_{m-1,0}(B_0) & \displaystyle\sum_{i=0}^{1}(-1)^{m-1+i}\mathcal{T}_{m-1-i,i}(B_{1-i}) & \displaystyle\sum_{i=0}^{2}(-1)^{m-1+i}\mathcal{T}_{m-1-i,i}(B_{2-i}) & \cdots & \displaystyle\sum_{i=0}^{n}(-1)^{m-1+i}\mathcal{T}_{m-1-i,i}(B_{n-i})& \displaystyle\sum_{i=1}^{m-1}(-1)^{m-1+i}\mathcal{T}_{m-1-i,i}(B_{n+1-i}) & \displaystyle\sum_{i=2}^{m-1}(-1)^{m-1+i}\mathcal{T}_{m-1-i,i}(B_{n+2-i}) & \displaystyle\sum_{i=3}^{m-1}(-1)^{m-1+i}\mathcal{T}_{m-1-i,i}(B_{n+3-i}) & \cdots & \mathcal{T}_{0,m-1}(B_n) \\
\end{pmatrix}$ 
\end{tabular}
}
\end{center}

\medskip

\noindent The first $n$ rows involve the $A_i$'s and the last $m$ rows involve the $B_j$'s.
\medskip

From the preceding $(m+n)\times(m+n)$ matrix $M$, we can define the left $(\sigma,\delta)$-resultant.

\begin{definition}\label{leftresultant}
 Let $f,g\in\mathcal{R}$ be skew polynomials of non-negative degrees $m$ and $n$, respectively, with $\sigma$ an automorphism. The above matrix $M$ will be called the \textit{left $(\sigma,\delta)$-Sylvester matrix} of $f$ and $g$, we will denote by $\text{Sylv}_{\mathbb{F},L}^{\sigma,\delta}(f,g)$. Finally, we define the \textit{left $(\sigma,\delta)$-\textit{resultant}} of $f$ and $g$ (over $\mathbb{F}$), denoted by $R^{\sigma,\delta}_{\mathbb{F},L}(f,g)$, 
 as the Dieudonn\'e determinant of $M$.
\end{definition}

\begin{remark}
If $\delta=0$, then formula \eqref{For3.3} can be written as $\mathcal{A}_i=\sigma^{-i}(a_i)$ for all $i=1,...,m$, and $\mathcal{A}_0=a_0$.
Hence $R^{\sigma,0}_{\mathbb{F},L}(f,g)=\text{Ddet}(M)$ with $$M= \begin{pmatrix}
a_{0} & \sigma^{-1}(a_{1}) & \sigma^{-2}(a_2) & \cdots  & \sigma^{-m}(a_m) & 0 & \cdots & 0\\
0 & \sigma^{-1}(a_{0}) & \sigma^{-2}(a_{1})  & \cdots & \sigma^{-m}(a_{m-1}) & \sigma^{-(m+1)}(a_m) & \cdots & 0\\
\vdots & \vdots & \ddots  & \vdots & \vdots & \vdots & \ddots & \vdots &\\
0 & 0 & 0 & \cdots & \sigma^{-(n-1)}(a_0) & \sigma^{-n}(a_{1}) & \cdots & \sigma^{-(n-1+m)}(a_m) \\
b_{0} & \sigma^{-1}(b_{1}) & \sigma^{-2}(b_{2}) & \cdots  & \sigma^{-n}(b_n) & 0 & \cdots & 0\\
0 & \sigma^{-1}(b_{0}) & \sigma^{-2}(b_{1})  & \cdots & \sigma^{-n}(b_{n-1}) & \sigma^{-(n+1)}(b_n) & \cdots & 0\\
\vdots & \vdots & \ddots  & \vdots & \vdots & \vdots & \ddots & \vdots &\\
0 & 0 & 0 & \cdots & \sigma^{-(m-1)}(b_0) & \sigma^{-m}(b_{1}) & \cdots & \sigma^{-(m-1+n)}(b_n) 
\end{pmatrix}$$
\end{remark} 
\smallskip

To obtain an algorithm that allows us to compute the left $(\sigma,\delta)$-Sylvester matrix (see Definition \ref{leftresultant}), we will need first the Algorithm \ref{Alg7} below.
 
\newpage 

\begin{algorithm}[h!]
 \small
\caption{Computation of $\mathcal{A}_i$}
\label{Alg7}
\begin{algorithmic}[1]
\Require{$f(x)=\sum_{i=0}^m a_ix^i$ and $i\in\{0,\dots,m\}$}
\Ensure{ $\mathcal{A}_i$}
\State{$A_i\gets 0$}
\For{$j\gets 0$ to $m+1-i$}
\State{$A_i\gets A_i+(-1)^j \cdot \mathcal{T}_{j,i}(a_{j+i-1})$}
\EndFor \\
\Return{$\mathcal{A}_i$}
\end{algorithmic}
\end{algorithm}

By Algorithms \ref{Alg2bis} and \ref{Alg7}, we can produce now the following algorithm which allows us to compute $\text{Sylv}_{\mathbb{F},L}^{\sigma,\delta}(f,g)$.

\begin{algorithm}[h!] 
\small
\caption{Computation of the left $(\sigma,\delta)$-Sylvester matrix of $f(x)=a_0+a_1x+\cdots+a_mx^m$ and $g(x)=b_0+b_1x+\cdots+b_nx^n$.}
\label{Algorithm:LefResultante}
\begin{algorithmic}[1]
\Require{$f,g\in\mathcal{R}$.}
\Ensure{Left $(\sigma,\delta)$-Sylvester matrix $M$ of $f$ and $g$.}
\State{$M_1\gets \begin{pmatrix}  \mathcal{A}_{0}  &   \mathcal{A}_{1} &\mathcal{A}_2 & \cdots & \mathcal{A}_{n+m} \end{pmatrix}$ }
\State{$M_2\gets \begin{pmatrix}  \mathcal{B}_{0}  &   \mathcal{B}_{1}&\mathcal{B}_2 & \cdots & \mathcal{B}_{n+m} \end{pmatrix}$ }
\For{$p\gets 1$ to $n-1$}
\State{$M_3\gets \begin{pmatrix}  (-1)^p\cdot\mathcal{T}_{p,0}(\mathcal{A}_{0}) \end{pmatrix}$}
\For{$q\gets 1$ to $n+m-1$}
\State{$Z_1\gets 0$}
\For{$l\gets 0$ to $p$}
\If{$0\leq q-l\leq m$}
\State{$Z_1 \gets Z_1 + \mathcal{T}_{p-l,l}(\mathcal{A}_{q-l})$}
\EndIf
%
%
%
%
\EndFor
\State{$M_3\gets \left( \begin{array}{c|c} M_3 &Z_1 \end{array}\right)$}
\EndFor
\State{$M_1\gets \left( \begin{array}{c} M_1 \\ \hline   M_3 \end{array}\right)$}
\EndFor 
 \For{$p\gets 1$ to $m-1$}
\State{$M_4\gets \begin{pmatrix}  \mathcal{T}_{p,0}(\mathcal{B}_{0}) \end{pmatrix}$}
\For{$q\gets 1$ to $n+m-1$}
\State{$Z_2\gets 0$}
\For{$l\gets 0$ to $p$}
\If{$0\leq q-l\leq n$}
\State{$Z_2 \gets Z_2 + \mathcal{T}_{p-l,l}(\mathcal{B}_{q-l})$}
\EndIf
\EndFor
\State{$M_4\gets \left( \begin{array}{c|c} M_4 &Z_2  \end{array}\right)$}
\EndFor
\State{$M_2\gets \left( \begin{array}{c} M_2 \\ \hline   M_4 \end{array}\right)$}
\EndFor

\algstore{myalg}
\end{algorithmic}
\end{algorithm}

\begin{algorithm}[h!]                     
\begin{algorithmic} [1]                  
\algrestore{myalg}

\State{$M\gets \left( \begin{array}{c} M_1 \\ \hline   M_2 \end{array}\right)$}\\
\Return{$M$}
\end{algorithmic}
\end{algorithm}

\medskip

As an application of the above algorithms, we can calculate in Magma the left
$(\sigma,\delta)$-Sylvester matrix of $f=x^2+wx$ and $g=x^2+w^2x+1$ in $\mathbb{F}_4[x;\sigma,\delta]$, where $\mathbb{F}_4=\{0,1,w,w^2\}, \sigma(a)=a^2, \delta(a)=w(\sigma(a)+a)$ for every $a\in\mathbb{F}_4$. Note that in this situation $\sigma^{-1}=\sigma$. 

First, write the following instructions in Magma:

\begin{verbatim}
F<w>:=GF(4);
S:= map< F -> F | x :-> x^2 >; 
D:= map< F -> F | x :-> w*(S(x)+x) >;
\end{verbatim}

Then, by typing the next Magma program

\begin{prog}\label{prog 9} ~
\begin{verbatim}
PosComT:=function(i,j,a)
C:= [u: u in [VectorSpace(GF(2),i+j)!v : v in 
VectorSpace(GF(2),i+j)]| Weight(u) eq i];
A:=0;
for k in [1..#C] do
 b:=a;
  for l in [1..i+j] do
   if C[k][l] eq 1 then
    b:=D(S(b));
    else
    b:=S(b); 
   end if;
  end for;
 A:=A+b;
end for;
return A;
end function;

Ai:=function(f,i)
A:=0;
for j in [0..#f-i] do
 A:=A+(-1)^(j)*PosComT(j,i-1,f[j+i]);
end for;
return A;
end function;

SumPosComT:=function(f,i,j)
AA:=0;
for k in [0..i-1] do
 if j-k ge 1 and j-k le #f then
  if i-1 ne 0 then
   AA:=AA+(-1)^(i-1+k)*PosComT(i-1-k,k,Ai(f,j-k));
   else
   AA:=(-1)^(i-1+k)*Ai(f,j-k);
  end if;
 end if;
end for;
return AA;
end function;

LeftSylvesterMatrix:=function(f,g)
m:=#f-1;
n:=#g-1;
if n ne 0 then
 M1:= Matrix(F,1,n+m,[SumPosComT(f,s,t): s in {1}, t in {1..n+m}]);
  for p in [2..n] do
   X:=Matrix(F,1,n+m,[SumPosComT(f,s,t): s in {p}, t in {1..n+m}]);
   M1:=VerticalJoin(M1,X);
  end for;
 else
 M1:=RemoveRow(ZeroMatrix(F,1,n+m),1);
end if;
if m ne 0 then
 M2:= Matrix(F,1,n+m,[SumPosComT(g,s,t): s in {1}, t in {1..n+m}]);
  for p in [2..m] do
   X:=Matrix(F,1,n+m,[SumPosComT(g,s,t): s in {p}, t in {1..n+m}]);
   M2:=VerticalJoin(M2,X);
  end for;
 else
 M2:=RemoveRow(ZeroMatrix(F,1,n+m),1);
end if;
M:=VerticalJoin(M1,M2);
return M;
end function;
\end{verbatim}
\end{prog}

\smallskip

and writing the following instruction

\begin{verbatim}
LeftSylvesterMatrix([0,w,1],[1,w^2,1]);
\end{verbatim}

we obtain
\begin{equation}\label{S}
S:=Sylv^{\sigma, \delta}_{\mathbb{F}_4,L}(f,g)=
\begin{pmatrix}
w & w^2 & 1 & 0 \\
w & 1 & w  & 1\\
w^2 & w & 1  & 0 \\
w & 0& w^2 & 1 
\end{pmatrix}\ .
\end{equation}

\medskip

\noindent By using $R^{\sigma,\delta}_{\mathbb{F},L}(f,g)$, we can give a left-hand version of Theorem \ref{Res1.5} as follows.

\begin{theorem}\label{teorizq}
Let $\sigma$ be an automorphism of $\mathbb{F}$ and let $f,g\in \mathcal{R}$ be two skew polynomials of positive degree $m$ and $n$, respectively. Then the following conditions are equivalent:
\begin{itemize}
    \item[$1)$]$\text{R}^{\sigma,\delta}_{\mathbb{F},L}(f,g)=0$;
    \item[$2)$] $f$ and $g$ have a common (non-unit) left factor in $\mathcal{R}$;
    \item[$3)$] $gcld(f,g)\neq 1$ (where "gcld" means greatest  common left  divisor);
    \item[$4)$] there are no polynomials $p,q \in \mathcal{R}$ such that $fp + gq = 1$;
    \item[$5)$] $f\mathcal{R}+g\mathcal{R}\subsetneq \mathcal{R}$.
\end{itemize}
\end{theorem}

\begin{proof}
Similar to Theorem \ref{Res1.5}.
\end{proof}

\begin{example}
Consider $\mathbb{F}_4[x;\sigma,\delta]$ with $\mathbb{F}_4=\{ 0,1,w, w^2\}$, where $w^2+w+1=0$, $\sigma(a)=a^2$ and $\delta(a)=w(\sigma(a)+a)$ for all $a \in \mathbb{F}_4$. In Example \ref{Res1.7} we have seen that given 
   $f:=(x+1)(x+w)=x^2+w x$  and $g:=(x+1)(x+w^2)=x^2+w^2 x+1$ 
we have $R^{\sigma,\delta}_{\mathbb{F}_4}(f,g)=w^2 \neq 0$, but 
$R^{\sigma,\delta}_{\mathbb{F}_4,L}(f,g)=\det S=0$ with $S$ as in \eqref{S},
according to Theorem \ref{teorizq}.
\end{example}

\noindent Let us continue here by giving left-hand versions of some previous results, whose proofs we omit because are similar to those of Theorem \ref{teor 3.7} and Proposition \ref{pro3.19}, respectively.

\begin{theorem}\label{}
Let $\sigma$ an automorphism of $\mathbb{F}$ and $\mathcal{P}_k(\mathbb{F})$ be the set of the
polynomials in $\mathcal{R}$ of degree less than or equal to $k$ with coefficients in $\mathbb{F}$. Let $f,g\in\mathcal{R}$ be two polynomials of positive degree $m,n$ respectively. Consider the right $\mathbb{F}$-linear map 
$$\varphi : \mathcal{P}_{n-1}(\mathbb{F})\oplus\mathcal{P}_{m-1}(\mathbb{F})\to\mathcal{P}_{n+m-1}(\mathbb{F})$$
defined by $\varphi ((a,b)):=fa+gb$. Then
$$\deg gcld(f,g)=\dim \ker \varphi=\dim \ker \phi = n+m-rr.rk(M)=n+m-lc.rk(M)  \ ,$$
where $\phi: \mathbb{F}^{n+m}\to\mathbb{F}^{n+m}$ is the right $\mathbb{F}$-linear map given by $\phi(\vec{x}):=M\cdot \vec{x}^T$ with $M:=Sylv_{\mathbb{F},L}^{\sigma,\delta}(f,g)$ the matrix defined in \eqref{For9} and $rr.rk(M)$ ($lc.rk(M)$) is the right row (left column) rank of $M$ which means the dimension of the $\mathbb{F}$-subspace spanned by the rows (columns) of $M$ viewed as elements of the $n+m$-dimensional right (left) vector space $\mathcal{P}_{n+m-1}(\mathbb{F})$ over $\mathbb{F}$.
\end{theorem}

\begin{proposition}\label{pro3.38}
Let $f,g\in\mathcal{R}$ be two skew polynomials of positive degree. Then, there are $A,B\in\mathcal{R}$ such that
$$fA+gB=R_{\mathbb{F},L}^{\sigma,\delta}(f,g) \ ,$$
where the coefficients of $A$ and $B$ $(\text{mod}\; [\mathbb{F}^*,\mathbb{F}^*])$ are integer polynomials in the entries of $\text{Sylv}_{\mathbb{F},L}^{\sigma,\delta}(f,g)$.
\end{proposition}
 
Moreover, we can reformulate Proposition \ref{basicpropert}, Theorem \ref{Res1.15} and Corollary \ref{Res1.16} as follows.
\begin{proposition} 
Let $\sigma$ be an automorphism of $\mathbb{F}$ and let $f,g\in \mathcal{R}$ be two skew polynomials of non-negative degree $m$ and $n$, respectively. The following properties hold:
\begin{itemize}
    \item[$1)$] $R_{\mathbb{F},L}^{\sigma,\delta}(g,f)=(-1)^{mn}R_{\mathbb{F},L}^{\sigma,\delta}(f,g)$. 
    \item[$2)$] $R_{\mathbb{F},L}^{\sigma,\delta}(-f,g)=(-1)^{n}R_{\mathbb{F},L}^{\sigma,\delta}(f,g)$ and $R_{\mathbb{F},L}^{\sigma,\delta}(f,-g)=(-1)^{m}R_{\mathbb{F},L}^{\sigma,\delta}(f,g)$.
    \item[$3)$] If $g=x-a$, then $R_{\mathbb{F},L}^{\sigma,\delta}(f,g)=0$ if and only if $f_L(a)=0$. In particular, if $a=0$ we have $R_{\mathbb{F},L}^{\sigma,\delta}(f,g)=f_L(0)\;(\text{mod}\;[\mathbb{F}^*,\mathbb{F}^*])$
    \item[$4)$] If $g=b_0$, then $R_{\mathbb{F},L}^{\sigma,\delta}(f,g)=b_0\sigma^{-1}(b_0)\sigma^{-2}(b_0)\cdots \sigma^{-(m-1)}(b_0)\;(\text{mod}\; [\mathbb{F}^*,\mathbb{F}^*])$.
    \end{itemize}
\end{proposition}
\begin{proof}
The proofs of the statements $1)$, $2)$ and $4)$ are similar to the proof of Proposition \ref{basicpropert}. Finally, statement $3)$ follows easily from equivalence between $1)$ and $2)$ of Theorem \ref{teorizq}.
\end{proof}

\begin{theorem}\label{res 3.41} 
Two non-constant skew polynomials $f,g\in \mathbb{F}_q[x;\sigma,\delta]$ have a common left root in some polynomial ring extension $\tilde{\mathbb{F}}[x;\Tilde{\sigma},\Tilde{\delta}]$ of $\mathbb{F}_q[x;\sigma,\delta]$ if and only if $\text{R}^{\sigma,\delta}_{\mathbb{F}_q,L}(f,g)=0$.
\end{theorem}
\begin{proof}
The left-to-right implication follows from Theorem \ref{teorizq} and the fact that $\text{R}^{\Tilde{\sigma},\Tilde{\delta}}_{\tilde{\mathbb{F}},L}(f,g)=\text{R}^{\sigma,\delta}_{\mathbb{F}_q,L}(f,g)$. Conversely, the proof it is analogous to the right-to-left implication of Theorem \ref{Res1.15} by using \cite[Theorem 3.2]{phdtesis} and by exchanging the functions $N^{\sigma,\delta}_i$ with $M^{\sigma,\delta}_i$ of Lemma \ref{evizquierda} together with slight modifications.
\end{proof}

\begin{corollary}
Let $\mathbb{F}$ be a division ring and let $f,g\in\mathcal{R}$ be two non-constant skew polynomials. If $f$ and $g$ have a common left root in some polynomial ring extension $\tilde{\mathbb{F}}[x;\Tilde{\sigma},\Tilde{\delta}]$ of $\mathcal{R}$, then $\text{R}^{\sigma,\delta}_{\mathbb{F},L}(f,g)=0$. 
\end{corollary}

\smallskip

\subsection{Right and left multiple roots}
In this subsection, under the assumption that $\sigma$ is an automorphism of $\mathbb{F}$, we will use the left and right $(\sigma,\delta)$-resultants to analyse the existence of right and left multiple roots of a skew polynomial $f\in\mathcal{R}$, respectively. 

First of all, let us give here some notions of derivatives of polynomials in $\mathcal{R}$ (see also \cite{DerEric} and \cite{zerosmartinez}).

\begin{definition}\label{derivadas}
Let $f\in\mathcal{R}$ with $\sigma$ an endomorphism (automorphism) of $\mathbb{F}$ and $a\in \mathbb{F}$. We define the \textit{first right (left) $(\sigma,\delta)$-derivative} of $f$ at $a$, with respect to the variable $x$, as the right (left) evaluation of $\Delta^{1}_{a}f(x)\in\mathcal{R}$ ($\Delta^{1}_{a,L}f(x)\in\mathcal{R}$) at the point $a$, where $\Delta^{1}_{a}f(x)$ ($\Delta^{1}_{a,L}f(x)$) is obtained by writing $f(x)=\Delta^{1}_{a}f(x)\cdot (x-a)+f(a)\ (f(x)=(x-a)\cdot \Delta^{1}_{a,L}f(x) +f_L(a))\ $.
The skew polynomial $\Delta^{1}_{a}f(x)$ ($\Delta^{1}_{a,L}f(x)$) will be called the \textit{first right (left) $(\sigma,\delta)$-derivative polynomial of $f$ by $a$}.
\end{definition}

By a recursive argument, from Definition \ref{derivadas} one can construct the right (left) $(\sigma,\delta)$-derivative polynomials of higher order of any polynomial in $\mathcal{R}$.

\begin{definition}\label{Hasse.martinez}
Let $\sigma$ be an endomorphism (automorphism) of $\mathbb{F}$. Given $f\in\mathcal{R}$, $r\in \mathbb{Z}_{>0}$ and a sequence ${\bf a}=(a_1,a_2,...,a_r)\in\mathbb{F}^r$, we define the  \textit{right (left) $(\sigma,\delta)$-derivative polynomial of $f$ of order $r$ via ${\bf a}$}, denoted by $\Delta_{{\bf a}}f(x)$ ($\Delta_{{\bf a},L}f(x)$) as the quotient upon right (left) division of $f$ by $P_{{\bf a}}:=(x-a_r)\cdots (x-a_2)(x-a_1)$ ($P_{{\bf a},L}:=(x-a_1) (x-a_2)\cdots (x-a_r)$). In particular, when $a=a_1=\dots =a_r$, we will simply write $\Delta_a^rf(x)$ ($\Delta_{a,L}^rf(x)$).
\end{definition}

\begin{remark}\label{remark4.3.9}
Let $\sigma$ be an endomorphism (automorphism) of $\mathbb{F}$ and consider
${\bf a}=(a_1,\dots,a_r)\in\mathbb{F}^r$. As in \cite{zerosmartinez}, one can define \textit{the right (left) Hasse derivative} $D_{{\bf a}}(f)\in\mathbb{F}$ ($D_{{\bf a},L}(f)\in\mathbb{F}$) \textit{of order} $r$ as the coefficient of the monomial of degree $r-1$ of the remainder in the right (left) division of $f$ by $P_{{\bf a}}$ ($P_{{\bf a},L}$) (see \cite[Definition 31 and Lemma 52]{zerosmartinez}). In this case, we have $D_{{\bf a}}(f)=\Delta_{{\bf a}'}f(a_r)$ ($D_{{\bf a},L}(f)=\left(\Delta_{{\bf a}',L}f\right)_L(a_r)$), where ${\bf a}'=(a_1,\dots,a_{r-1})$.
\end{remark}

Let $f(x)=\sum_{i=0}^m\alpha_i x^{i}$ and $\Delta_a^{1}f(x)=\sum_{j=0}^{m-1}\beta_j x^{j}$ be skew polynomials in $\mathcal{R}$ as in  Definition \ref{derivadas} for some $a\in\mathbb{F}$. By Lemma \ref{LemaCij}, we have
$$f(x)-f(a)=\Delta^{1}_{a}f(x)\cdot (x-a) \Longleftrightarrow \sum_{i=0}^m\alpha_i x^{i}-f(a)= \sum_{i=0}^{m-1} \left(\beta_ix^{i+1}-\sum_{k=0}^{i}\beta_i\;  \mathcal{C}_{k,i-k}(a)x^{i-k}\right). $$ 
Then, comparing the coefficients of the positive powers $x^t$ in the latest equality, we get the following recursive formula:
\begin{equation}\label{(5)}
     \beta_{m-1} = \alpha_m \ ,\ \ \
 \beta_k =\alpha_{k+1}+\sum_{i=0}^{m-k-2} \beta_{k+1+i}\;\mathcal{C}_{i,k+1}(a) \qquad \forall k=m-2,m-3,...,0.
\end{equation}

\noindent Using \eqref{(5)}, the next algorithm shows how to compute $\Delta_{{\bf a}}f(x),\ {\bf a}=(a_1,\dots,a_n)\in \mathbb{F}^{n}$.

\begin{algorithm}[h!]
 \small
\caption{Computation of $\Delta_{{\bf a}}f(x)\in\mathcal{R}$.}
\label{Algorithm:Dx}
\begin{algorithmic}[1]
\Require{$f(x)=\sum_{i=0}^m \alpha_ix^i\in\mathcal{R}$, ${\bf a}=(a_1,\dots,a_n)\in\mathbb{F}^n$, $n\in \mathbb{Z}_{\geq 1}$ and $n\leq m$}
\Ensure{ $\Delta_{{\bf a}}f(x)$}
\For{$i\gets 1$ to $n$}
\State{$m\gets \deg f(x)$}
%
%
%
%
%
\State{$\beta_{m-1}\gets \alpha_{m}$}
\For{$h\gets 0$ to $m-2$}
\State{$s\gets 0$}
\For{$j\gets 0$ to $h$}
\State{$s_1\gets \beta_{m-1-h+j}\cdot \mathcal{C}_{j,m-1-h}(a_i)$}
\State{$s\gets s+s_1$}
\EndFor
\State{$\beta_{m-2-h}\gets \alpha_{m-1-h}+s$}
\EndFor
\State{$q(x)\gets \sum_{t=0}^{m-1} \beta_tx^t$}
\State{$f(x)\gets q(x)$}
\EndFor\\
\Return{$q(x)$}
\end{algorithmic}
\end{algorithm}

Note that, when $\sigma$ is an automorphism, similar accounts as above can be made with $f(x)-f_L(a)=(x-a)\cdot \Delta_{a,L}^1f(x)$. Moreover, as an application of Algorithm \ref{Algorithm:Dx}, we apply the next Magma program to compute $\Delta_{(a,a)}f$ when $f=x^4-jx^2+(2i-k) \in \mathbb{H}[x;\sigma,0]$, $a=1+j$ and $\sigma(h):=ihi^{-1}$ for all $h\in \mathbb{H}$.
To do this, begin by writing in Magma the following instructions:
\medskip
\begin{verbatim}
F<i,j,k> := QuaternionAlgebra< RealField() | -1, -1 >;
R<x>:=PolynomialRing(F);
S:= map< F -> F | x :-> i*x*(1/i) >;
D:= map< F -> F | x :-> 0 >;
\end{verbatim}
\medskip
then, using the function ``PosCom" defined in Program \ref{prog 1}, we can continue with the following instructions to define a new function ``DerNA".

\newpage

\begin{prog} \label{program:Dx} ~
\begin{verbatim}
DerNA:=function(f,A)
 t:=#f;
 if #A ge t then
  f:=F!0;
 end if;  
 if #A le t-1 then
  if t eq 2 then
   f:=F!f[t];
  end if;
  if t ge 3 then
   for i in [1..#A] do
    t:=#f;
    b:=[ f[t] ];
    for h in [0..t-3] do
     s:=F!0;
     for j in [0..h] do
      s1:=b[h+1-j]*PosCom(j,t-2-h,A[i]);
      s:= s + s1;
     end for;
     b:= b cat [ f[t-1-h] + s ];
    end for;
    g:=[];
    for k in [1..#b] do
     g:=g cat [ b[#b+1-k] ];
    end for;
    f:=g;
   end for;
  end if;
 end if;
 return R!f;
end function;
\end{verbatim}
\end{prog}

Thus, typing in Magma 

\begin{verbatim}
DerNA([2*i-k,0,-j,0,1],[1+j,1+j]);
\end{verbatim}

\noindent we get  $\Delta_{(1+j,1+j)}(x^4-jx^2+2i-k)= x^2 + 2x + 4 - 3j$.
 
\medskip

\begin{remark}
Consider $f,g \in\mathcal{R}$, $\lambda,a \in \mathbb{F}$ and suppose that $\sigma$ is an endomorphism (automorphism) of $\mathbb{F}$. Then the linearity of $\Delta_{{\bf a}}$ and the fact that for any $a\in\mathbb{F}$
$$\Delta^{1}_{a}(f\cdot g)= f \cdot \Delta^{1}_{a}g + \Delta^{1}_{a}(f\cdot g(a))\ \ \ (\Delta^{1}_{a,L}(f\cdot g)=\Delta^{1}_{a,L}f\cdot g + \Delta^{1}_{a,L}(f_L(a) \cdot g))$$ 
allow one to obtain recursive formulas for right (left) $(\sigma,\delta)$-derivative polynomials of order $r\in\mathbb{Z}_{\geq 1}$ via ${\bf a}=(a_1,\dots,a_r)\in\mathbb{F}^r$. For instance, given any $n\in\mathbb{Z}_{\geq 1}$, we get $\Delta^{1}_{a}x^{n} = x^{n-1} + \sum_{k=0}^{n-1} C_{k,n-1-k} (a) \cdot \Delta^{1}_{a}x^{n-1-k} \ \ \ \left(\Delta^{1}_{a,L}x^{n} = x^{n-1}+\Delta_{a,L}^1\left( \sum_{k=0}^{n-1}x^{n-1-k}(-1)^k\mathcal{T}_{k,n-1-k}(a)  \right)\right)$.
\end{remark}

\medskip

Finally, consider the next classical definition of right (left) multiplicity of roots.

\begin{definition} \label{Res2.1}
Consider $f\in\mathcal{R}$, $a\in \mathbb{F}$ and $r\in \mathbb{Z}_{\geq 1}$. If $\sigma$ is an endomorphism (automorphism), we say that $a$ is a \textit{right (left) root
of $f$ of multiplicity $\geq r$} if the skew polynomial $(x-a)^{r}$ divides $f$ on the right (left). Moreover, we say that $a$ is a right (left) root of $f$ of multiplicity $r$ if the skew polynomial $(x-a)^{r}$ is the maximum power of $x-a$ which divides $f$ on the right (left).
\end{definition}

\begin{example}
Let $\mathbb{F}_9[x;\sigma,0]$ with $\sigma(z):=z^3$ for all $z\in\mathbb{F}_9$. If $x=a\in\mathbb{F}_9$ is a right root of $g(x)\in\mathbb{F}_9[x;\sigma,0]$ of multiplicity $\geq 2$, then $\text{R}^{\sigma,\delta}_{\mathbb{F}_9}(g,\Delta_{a}^{1}g)=0$. On the other hand, consider $f(x)=(x+1)(x-1)\in \mathbb{F}_9[x;\sigma,0]$. Then $\Delta_{1}^1f(x)=x+1$ and $R^{\sigma,0}_{\mathbb{F}_9}(f,\Delta_1^1f)=0$, because we can write $f(x)=(x-1)(x+1)$, but $x=1$ is a right root of $f(x)$ of multiplicity one.
\end{example}

Keeping in mind all the previous definitions, we obtain the following result.

\begin{theorem}\label{Res2.2}
Consider $f\in\mathcal{R}$, $a\in \mathbb{F}$ and $r$ a positive integer such that $r<\deg f$. If $\sigma$ is an automorphism of $\mathbb{F}$, then the following are equivalent:
\begin{itemize}
\item[$1)$] $a$ is a right (left) root of $f$ of multiplicity $\geq r$; 
\item[$2)$] $a$ is a common right (left) root of $f,\Delta_{a}^{1}f,\dots,\Delta_{a}^{r-1}f \ \left(f,\Delta_{a,L}^{1}f,\dots,\Delta_{a,L}^{r-1}f\right)$;
\item[$3)$] $\text{R}^{\sigma,\delta}_{\mathbb{F}, L}(\Delta_{a}^{j}f,\Delta_{a}^{j+1}f)=0 \ \left(\text{R}^{\sigma,\delta}_{\mathbb{F}}(\Delta_{a,L}^{j}f,\Delta_{a,L}^{j+1}f)=0 \right)$ for $j=0,\dots,r-1$;
\item[$4)$] $gcld(\Delta_{a}^{j}f,\Delta_{a}^{j+1}f)\neq 1 \ \left(gcrd(\Delta_{a,L}^{j}f,\Delta_{a,L}^{j+1}f)\neq 1\right)$ for $j=0,\dots,r-1$,
\end{itemize}
where $\Delta_{a}^{0}f(a):=f(a) \ \left(\left(\Delta_{a,L}^{0}f\right)_L(a):=f_L(a)\right)$.
\end{theorem}

\begin{proof}
The equivalence between $1)$ and $2)$ follows from 
Definition \ref{Res2.1} and the equalities
$$\Delta_{a}^{i}\left(g(x)(x-a)^t\right)=g(x)(x-a)^{t-i} \ \ \ (\Delta_{a,L}^{i}\left((x-a)^tg(x)\right)=(x-a)^{t-i}g(x))\ ,$$
$$(*)\ \ \Delta_{a}^{i}f(x)=\Delta_{a}^{i+1}f(x)(x-a)+\Delta_{a}^{i}f(a) \ \ \ (\Delta_{a,L}^{i}f(x)=(x-a)\Delta_{a,L}^{i+1}f(x)+\left(\Delta_{a,L}^{i}f\right)_L(a)) \ ,$$
for $i=0,\dots,r-1$, where $\Delta_a^0f(x)=f(x) \ \left( \Delta_{a,L}^0f(x)=f(x) \right)$, while the equivalences $2) \Leftrightarrow 3) \Leftrightarrow 4)$ follow from Theorems \ref{Res1.5} and \ref{teorizq}, and the fact that for every $j=0,\dots,r-1$, we have $\Delta_{a}^{j}f(a)=0\ (\Delta_{a,L}^{j}f(a)=0) \iff \text{R}^{\sigma,\delta}_{\mathbb{F}, L}(\Delta_{a}^{j}f,\Delta_{a}^{j+1}f)=0 \ (\text{R}^{\sigma,\delta}_{\mathbb{F}}(\Delta_{a,L}^{j}f,\Delta_{a,L}^{j+1}f)=0)$ by the equations $(*)$.
\end{proof}
 
Recently, in \cite{zerosmartinez} the author proposed a definition of multiplicity distinct from the previous one, but which coincide in the commutative case (that is, when $\mathbb{F}$ is a field, $\sigma=Id$ and $\delta=0$).

\begin{definition}\label{mult.martinez}
Let $\sigma$ be an endomorphism (automorphism) of $\mathbb{F}$. For $r\in \mathbb{Z}_{>0}$, we say that a sequence ${\bf a} = (a_1,a_2,...,a_r)\in \mathbb{F}^r$ is a \textit{right (left) $(\sigma,\delta)$-multiplicity sequence} if $a_1$ is the only right (left) root of the skew polynomial $P_{{\bf a}}$ ($P_{{\bf a},L}$). Moreover, given $f\in\mathcal{R}$, $r\in \mathbb{Z}_{>0}$ and a right (left) $(\sigma,\delta)$-multiplicity sequence ${\bf a}=(a_1,...,a_r)\in \mathbb{F}^{r}$ as before, we say that $a_1$ is a \textit{right (left) zero of $f$ of multiplicity $r$ via ${\bf a}$} if the skew
polynomial $P_{{\bf a}}$ ($P_{{\bf a},L}$) divides $f$ on the right (left).
\end{definition}

Finally, with this new notion of multiplicity, we get also the next result.

\begin{theorem}\label{thm-multiplicity}
Let $\sigma$ be an automorphism of $\mathbb{F}$. Consider $f\in \mathcal{R}$, $a\in \mathbb{F}$, $r$ a positive integer such that $r<\deg f$ and ${\bf a}=(a_1,...,a_r)\in \mathbb{F}^{r}$ a right (left) $(\sigma,\delta)$-multiplicity sequence. Then the following are equivalent:
\begin{itemize}
\item[$1)$] $a_1$ is a \textit{right (left) root of $f$ of multiplicity $r$ via ${\bf a}$}; 
\item[$2)$] $\Delta_{{\bf{a_i}}}f(a_{i+1})=0 \ \left(\left(\Delta_{{\bf{a_i}},L}f\right)_L(a_{i+1})=0 \right)$ for all $i=0,1,...,r-1$, where ${\bf{a_j}}=(a_1,\dots,a_{j})$ for $j=1,2,...,r-1$, $\Delta_{{\bf{a_0}}}f(a_1):=f(a_1) \ \left(\left(\Delta_{{\bf{a_0}},L}f\right)_L(a_1):=f_L(a_1)\right)$; 
\item[$3)$] $\text{R}^{\sigma,\delta}_{\mathbb{F}, L}(\Delta_{{\bf{a_i}}}f,\Delta_{{\bf{a_{i+1}}}}f)=0 \ \left(\text{R}^{\sigma,\delta}_{\mathbb{F}}(\Delta_{{\bf{a_i}},L}f,\Delta_{{\bf{a_{i+1}}},L}f)=0\right)$ for all $i=0,1,...,r-1$ ;
\item[$4)$] $gcld(\Delta_{{\bf{a_i}}}f,\Delta_{{\bf{a_{i+1}}}}f)\neq 1 \ \left(gcrd(\Delta_{{\bf{a_i}},L}f,\Delta_{{\bf{a_{i+1}}},L}f)\neq 1\right)$ for all $i=0,1,...,r-1$.
\end{itemize}
\end{theorem}

\begin{proof}
The equivalence between $1)$ and $2)$ follows from \cite[Proposition 45]{zerosmartinez} (a left-hand version of \cite[Proposition 45]{zerosmartinez} with suitable modifications) and Remark \ref{remark4.3.9}, while the equivalences $2) \Leftrightarrow 3) \Leftrightarrow 4)$ follow from Theorem \ref{teorizq} (Theorem \ref{Res1.5}) and the fact that for every $i=0,\dots,r-1$, we have 
$\Delta_{{\bf{a_i}}}f(a_{i+1})=0\ \left(\left(\Delta_{{\bf{a_i}},L}f\right)_L(a_{i+1})=0\right) \iff \text{R}^{\sigma,\delta}_{\mathbb{F}, L}(\Delta_{{\bf{a_i}}}f,\Delta_{{\bf{a_{i+1}}}}f)=0 \ \left(\text{R}^{\sigma,\delta}_{\mathbb{F}}(\Delta_{{\bf{a_i}},L}f,\Delta_{{\bf{a_{i+1}}},L}f)=0\right)$ because $\Delta_{{\bf{a_i}}}f(x)=\Delta_{{\bf{a_{i+1}}}}f(x)(x-a_{i+1})+\Delta_{{\bf{a_i}}}f(a_{i+1}) \ \ \left(\Delta_{{\bf{a_i}},L}f(x)=(x-a_{i+1})\Delta_{{\bf{a_{i+1}}},L}f(x)+\left(\Delta_{{\bf{a_i}},L}f\right)_L(a_{i+1})\right)$.
\end{proof}

\vspace{0.5cm}

\noindent {\bf Acknowledgement.} This work 
was partially supported by the Proyecto VRID N. 219.015.023-INV of the University of Concepci\'on.

\end{document}